\documentclass[a4paper]{amsart}

\usepackage[utf8]{inputenc}
\usepackage[top=2.5cm, bottom=2.5cm, left=2.5cm, right=2.5cm]{geometry}
\usepackage{amsthm,amssymb,amsmath}
\usepackage{array,multirow}
\usepackage{cite}
\usepackage{float}
\usepackage{comment}
\usepackage[all,cmtip]{xy}

\usepackage{xcolor}
\usepackage[colorlinks=true,citecolor=blue,plainpages=false,pdfpagelabels]{hyperref}

\usepackage{tikz}

\def\genstepset#1#2#3#4#5#6#7#8#9{%
\begin{tikzpicture}[scale=#9, baseline=(current bounding box.center)]
\foreach \x in {-1,0,1} \foreach \y in {-1,0,1} \fill(\x,\y) circle[radius=2pt];
\ifx1#1\draw[thick,->](0,0)--(0,1);\else\fi
\ifx1#2\draw[thick,->](0,0)--(1,1);\else\fi
\ifx1#3\draw[thick,->](0,0)--(1,0);\else\fi
\ifx1#4\draw[thick,->](0,0)--(1,-1);\else\fi
\ifx1#5\draw[thick,->](0,0)--(0,-1);\else\fi
\ifx1#6\draw[thick,->](0,0)--(-1,-1);\else\fi
\ifx1#7\draw[thick,->](0,0)--(-1,0);\else\fi
\ifx1#8\draw[thick,->](0,0)--(-1,1);\else\fi
\end{tikzpicture}
}
\def\stepset#1#2#3#4#5#6#7#8{\raisebox{1.5mm}{\genstepset{#1}{#2}{#3}{#4}{#5}{#6}{#7}{#8}{.2}}}
\def\smallstepset#1#2#3#4#5#6#7#8{\raisebox{1.2mm}{\genstepset{#1}{#2}{#3}{#4}{#5}{#6}{#7}{#8}{.175}}}

\def\bK{\mathbb K}
\def\cS{\mathcal S}

\def\plaincomma{,} % HACK! WATCH OUT FOR UNEXPECTED SIDE EFFECTS
\catcode`,=\active
\let,\plaincomma

\def\pFqnoargs#1#2{{}_#1F_#2}
\def\pFq#1#2#3#4#5{
  \mathchoice
      {\pFqnoargs{#1}{#2}\biggl(\begin{matrix}{\def,{\kern.707em}#3}\\{\def,{\kern.707em}#4}\end{matrix}\,\bigg|\,#5\biggr)} % displaystyle
      {\pFqnoargs{#1}{#2}(#3;#4;#5)} % textstyle
      {\pFqnoargs{#1}{#2}(#3;#4;#5)} % scriptstyle
      {\pFqnoargs{#1}{#2}(#3;#4;#5)} % scriptscriptstyle
}
\def\twoFone#1#2#3#4{\pFq21{#1,#2}{#3}{#4}}

\makeatletter
\def\timenow{\@tempcnta\time
  \@tempcntb\@tempcnta
  \divide\@tempcntb60
  \ifnum10>\@tempcntb0\fi\number\@tempcntb
  \multiply\@tempcntb60
  \advance\@tempcnta-\@tempcntb
  :\ifnum10>\@tempcnta0\fi\number\@tempcnta}
\makeatother

\theoremstyle{definition}\newtheorem{definition}{Definition}
\theoremstyle{plain}\newtheorem{lem}{Lemma}
\theoremstyle{plain}\newtheorem{thm}{Theorem}
\theoremstyle{plain}\newtheorem{conj}{Conjecture}

\newcommand{\Diag}{\operatorname{Diag}}
\newcommand{\Res}{\operatorname{Res}}
\newcommand{\pp}{\operatorname{pp}}

\newcommand{\BMnM}{Bousquet-Mélou and Mishna}

\def\Rplus{\set R_{\geq0}}
\def\Rminus{\set R_{\leq0}}

\def\odotkn{\underset{k\leq n}{\odot}}
\def\odotnk{\underset{n\geq k}{\odot}}
\def\starkn{\underset{k\leq n}{\star}}
\def\starnk{\underset{n\geq k}{\star}}

\let\set\mathbb

\begin{document}

\title[Generating Functions of Walks in the Quarter Plane]{Hypergeometric Expressions for Generating Functions of~Walks with Small Steps in the Quarter Plane}

\author{Alin Bostan}
\address{INRIA, Université Paris-Saclay, 91120 Palaiseau, France}
\email{alin.bostan@inria.fr}

\author{Frédéric Chyzak}
\address{INRIA, Université Paris-Saclay, 91120 Palaiseau, France}
\email{frederic.chyzak@inria.fr}

\author{Mark van Hoeij}
\address{Department of Mathematics, Florida State University, Tallahassee FL 32306--4510}
\email{hoeij@math.fsu.edu}
\thanks{M.~van Hoeij's work was supported by NSF grant 1319547.}

\author{Manuel Kauers}
\address{Institute for Algebra, Johannes Kepler University, A4040 Linz}
\email{manuel.kauers@jku.at}
\thanks{M.~Kauers's work was supported by FWF grants Y464-N18, F5004 and W1214.}

\author{Lucien Pech}
% FC: ``address'' is really meant for affiliations. LP wants to show none.
%\address{???}
\email{lucien.pech@gmail.com}
\thanks{L.~Pech's work was supported in part by the Microsoft Research -- INRIA Joint Centre and by INRIA}

\date{\today\ --\ \timenow}

\begin{abstract} We study nearest-neighbors walks on the two-dimensional
square lattice, that is, models of walks on $\mathbb{Z}^2$ defined by a fixed
step set that is a subset of the non-zero vectors with coordinates 0, 1 or~$-1$.
We concern ourselves with the enumeration of such walks starting at the
origin and constrained to remain in the quarter plane $\mathbb{N}^2$, counted
by their length and by the position of their ending point. Bousquet-Mélou and
Mishna [Contemp.\ Math., pp.~1--39, Amer.\ Math.\ Soc., 2010] identified
19~models
of walks that possess a D-finite generating function; linear
differential equations have then been guessed in these cases by Bostan and
Kauers [FPSAC 2009, Discrete Math.\ Theor.\ Comput.\ Sci.\ Proc., pp.~201--215,
2009]. We give here the first proof that these equations are indeed satisfied
by the corresponding generating functions. As a first corollary, we prove that
all these 19 generating functions can be expressed in terms of Gauss'
hypergeometric functions that are intimately related to elliptic integrals. As
a second corollary, we show that all the 19 generating functions are
transcendental, and that among their $19 \times 4$ combinatorially meaningful
specializations only four are algebraic functions.
\end{abstract}

\maketitle


\smallskip\noindent{{\bf MSC2010}: Primary 05A15, 14N10, 33F10, 68W30; Secondary 33C05, 97N80, 11J89.

\section{Introduction}

\paragraph{Context.} An important problem in enumerative combinatorics is the
study of lattice walks in restricted lattices. Many efforts have been deployed
in recent years for classifying them, see e.g.\ the surveys~\cite{Humphreys10}
and~\cite{Krattenthaler15} and the references therein.
The generating functions of lattice walks
are not only intriguing for combinatorial reasons, but also from the
perspective of computer algebra. For combinatorial reasons they are
interesting because, depending on the choice of admissible steps, the
generating functions may have quite different algebraic and analytic
properties. For computational reasons they are interesting because their
descriptions (whether by a polynomial or by a linear differential equation, as
we will see below) are sometimes so large in size that it becomes difficult to
handle them with a reasonable efficiency.

In the present article, we consider \emph{small step walks restricted to the
quarter plane}, defined as follows. Let $\cS\neq \emptyset$ be a fixed subset
of~$\{-1,0,1\}^2\setminus\{(0,0)\}$, which will contain all steps allowed in
the walks. An $\cS$-walk of length $n$ starts at the origin $(0,0)$ and
consists of $n$ consecutive steps, where a step from a point $A$ to a point
$B$ is admitted if $B-A\in\cS$, and both $A$ and $B$ belong to the quarter
plane $\set N^2$. These walks are called \emph{restricted\/} to the quarter
plane because they are not allowed to step out of it, and \emph{with small
steps\/} because a single step changes the position by no more than $1$ in
each coordinate. As an example, for $\cS=\{(1,1),(-1,0),(0,-1)\}$
(\emph{Kreweras walks}), a possible walk of length six is
\[
  (0,0)\to (1,1) \to (2,2) \to (1,2) \to (2,3) \to (2,2)\to(2,1).
\]
A brute-force enumeration with rejection shows that, altogether, there are 125
different walks of length six for this particular step set. In the general
case of an arbitrary step set $\cS$, with $q_n$ denoting the number of
different $\cS$-walks of length~$n$, we are interested in the generating
function $Q(t):=\sum_{n=0}^\infty q_nt^n\in\set Q[[t]]$.

The generating function $Q(t)$ corresponding to the example step set $\cS$
above is \emph{algebraic}~\cite{kreweras65,Gessel86,BM05}, i.e., it satisfies
a polynomial equation $P(t,Q(t))=0$ for some $P\in\set Q[t,T]\setminus\{0\}$.
But this is not the case for all other step sets.
Still, among those step sets that induce
a \emph{transcendental} (i.e., non-algebraic) generating function~$Q(t)$,
some have a~$Q(t)$ that is \emph{D-finite}, i.e.,
that satisfies a linear differential equation with polynomial coefficients.
The step
set $\cS=\{(1,1),(-1,1),(0,-1)\}$ is an example for this
case~\cite{mbm02,BoPe03}. Finally, there are also step sets whose
corresponding generating function is not even D-finite; Mishna and
Rechnitzer~\cite{rechnitzer09} proved that this is the case for example when
$\cS=\{(-1,1),(1,1),(1,-1)\}$.

More generally, for a fixed step set~$\cS$, one is interested in the study of
the trivariate power series
\[
  Q(x,y;t)=\sum_{n=0}^\infty\sum_{i,j=0}^\infty q_{i,j;n}x^i y^j t^n,
\]
where $q_{i,j;n}$ denotes the number of $\cS$-walks of
length~$n$ starting at $(0,0)$ and ending at $(i,j)$. The power series
$Q(x,y;t)$ is called the \emph{complete generating function} for $\cS$-walks.
Note that the counting series~$Q(t)$ introduced before is nothing but the
specialization $Q(1,1;t)$ of $Q(x,y;t)$ at $(x,y)=(1,1)$. Other
combinatorially meaningful specializations are $Q(0,0;t)$, the generating
function of $\cS$-walks returning to the origin (also called
\emph{excursions}), $Q(1,0;t)$, the generating function of $\cS$-walks ending
on the horizontal axis, and $Q(0,1;t)$, the generating function of $\cS$-walks
ending on the vertical axis.

\BMnM~\cite{BMM} have undertaken a systematic classification of the 256 step
sets $\cS\subseteq\{-1,0,1\}^2\setminus\{(0,0)\}$, from the viewpoint of
structural properties of the generating function $Q(x,y;t)$. Again, the concerned properties are algebraicity and D-finiteness, yet applied to a multivariate setting\footnote{A trivariate power series
$S \in \set Q[[x,y,t]]$ is \emph{algebraic} if it is the root of a nonzero polynomial $P \in \set Q[x,y,t,T]$. It is called \emph{D-finite} if
the set of all partial derivatives of $S$ spans a finite-dimensional vector space over~$\mathbb{Q}(x,y,t)$.
}.
They found out that there are 79 inherently different and nontrivial models to
consider, of which they recognized 22 models for which $Q(x,y;t)$ is D-finite;
a 23rd model, of the so-called \emph{Gessel walks}, was proved to be D-finite
(and even algebraic) by Bostan and Kauers~\cite{BoKa10}. These 23 models share
the feature that a certain group associated to $\cS$ is finite. In the
remaining 56 cases, it has been proved that the groups are
infinite~\cite{BMM}, and that the complete generating functions $Q(x,y;t)$ are
not D-finite~\cite{rechnitzer09,KuRa12,BoKaSa14,MeMi14b}.

Of the 23 D-finite generating functions, 4~were recognized to be
algebraic~\cite{BMM,BoKa10}: one corresponds to the Kreweras model
$\cS=\{(1,1),(-1,0),(0,-1)\}$, one to its reverse
$\cS^{\textrm{rev}}=\{(-1,-1),(1,0),(0,1)\}$, one to $\cS \cup
\cS^{\textrm{rev}}$, and one to Gessel's.
All the other 19 generating functions~were proved to be
transcendental~\cite{FaRa10}, although some of their specializations are
algebraic, e.g., for $\cS=\{(-1,0),(0,1),(1,-1)\}$~\cite{BMM}. These 19
models form the main object of this article; they are depicted in the third
column of Table~\ref{tab:steps}.

\medskip
\begin{table}[h]
\begin{center}
\begin{tabular}{|@{\ }c@{\ }|@{\ }c@{\ }|c|c|c|}
\hline
   & OEIS \cite{OEIS} & $\cS$ & $N(x,y)$ & $S(x,y)$\\
\hline
1  & A005566 & \stepset10101010 & \multirow{4}{*}{$\displaystyle{\left(x-\bar{x}\right)\left(y-\bar{y}\right)}$} & $y+\left(x+\bar{x}\right)+\bar{y}$\\
2  & A018224 & \stepset01010101 & & $(x+\bar{x})y+(x+\bar{x})\bar{y}$\\
3  & A151312 & \stepset11011101 & & $(x+1+\bar{x})y+(x+1+\bar{x})\bar{y}$ \\
4  & A151331 & \stepset11111111 & & $(x+1+\bar{x})y+(x+\bar{x})+(x+1+\bar{x})\bar{y}$ \\
\hline
5  & A151266 & \stepset01001001 & \multirow{2}{*}{$\displaystyle{\left(x-\bar{x}\right)\left(y-\frac{1}{x+\bar{x}}\bar{y}\right)}$} & $(x+\bar{x})y+\bar{y}$ \\
6  & A151307 & \stepset01101011 & & $(x+\bar{x})y+(x+\bar{x})+\bar{y}$ \\
\hline
7  & A151291 & \stepset11001001 & \multirow{2}{*}{$\displaystyle{\left(x-\bar{x}\right)\left(y-\frac{1}{x+1+\bar{x}}\bar{y}\right)}$} & $(x+1+\bar{x})y+\bar{y}$ \\
8  & A151326 & \stepset11101011 & & $(x+1+\bar{x})y+(x+\bar{x})+\bar{y}$ \\
\hline
9  & A151302 & \stepset11010101 & \multirow{2}{*}{$\displaystyle{\left(x-\bar{x}\right)\left(y-\frac{x+\bar{x}}{x+1+\bar{x}}\bar{y}\right)}$} & $(x+1+\bar{x})y+(x+\bar{x})\bar{y}$ \\
10 & A151329 & \stepset11110111 & & $(x+1+\bar{x})y+(x+\bar{x})+(x+\bar{x})\bar{y}$ \\
\hline
11 & A151261 & \stepset10011100 & \multirow{2}{*}{$\displaystyle{\left(x-\bar{x}\right)\left(y-\left(x+1+\bar{x}\right)\bar{y}\right)}$} & $y+(x+1+\bar{x})\bar{y}$ \\
12 & A151297 & \stepset10111110 & & $y+(x+\bar{x})+(x+1+\bar{x})\bar{y}$ \\
\hline
13 & A151275 & \stepset01011101 & \multirow{2}{*}{$\displaystyle{(x-\bar{x})\left(y-\frac{x+1+\bar{x}}{x+\bar{x}}\bar{y}\right)}$} & $(x+\bar{x})y+(x+1+\bar{x})\bar{y}$ \\
14 & A151314 & \stepset01111111 & & $(x+\bar{x})y+(x+\bar{x})+(x+1+\bar{x})\bar{y}$ \\
\hline
15 & A151255 & \stepset10010100 & \multirow{2}{*}{$\displaystyle{\left(x-\bar{x}\right)\left(y-\left(x+\bar{x}\right)\bar{y}\right)}$} & $y+(x+\bar{x})\bar{y}$\\
16 & A151287 & \stepset10110110 & & $y+(x+\bar{x})+(x+\bar{x})\bar{y}$ \\
\hline
17 & A001006 & \stepset10010010 & \multirow{2}{*}{$\displaystyle{xy-\bar{x}y^2+\bar{x}^2y-\bar{x}\bar{y}+x\bar{y}^2-x^2\bar{y}}$}
& $y+\bar{x}+x\bar{y}$ \\
18 & A129400 & \stepset10111011 & & $(1+\bar{x})y+(x+\bar{x})+(1+x)\bar{y}$
\\
\hline
19 & A005558 & \stepset00110011 & $\displaystyle{xy-\bar{x}y^2+\bar{x}^3y^2-\bar{x}^3y+\bar{x}\bar{y}-x\bar{y}^2+x^3\bar{y}^2-x^3\bar{y}}$
& $\bar{x}y+(x+\bar{x})+x\bar{y}$
 \\
\hline
\end{tabular}
\caption{The 19 models ($\bar x = x^{-1}$, \ $\bar y = y^{-1}$) and the corresponding rational functions $N$ and $S$ in~Eq.~\eqref{eq:Q-as-pos-part}. Cases 1--16 are numbered as in~\cite[Table~1]{BMM}. Cases 17~and~18 are \#1~and~\#2 in~\cite[Table~2]{BMM}. Case~19 is \#1~in~\cite[Table~3]{BMM}. The OEIS tags correspond to coefficient sequences of the length generating function~$Q(t) = Q(1,1;t)$.}
\label{tab:steps}
\end{center}
\end{table}

The proofs of D-finiteness given by \BMnM\ are implicit (i.e., qualitative):
the existence of differential equations satisfied by the
generating functions $Q(x,y;t)$ was shown without obtaining the differential
equations explicitly. On the other hand, Bostan and Kauers~\cite{BK,BoKa}
provided explicit differential equations for (specializations of) the 23
D-finite generating functions, but these equations were determined only
experimentally and, in most of the transcendental cases, they still lack
formal proofs (in the 4 algebraic cases, differential equations are easily
proved, starting from algebraic equations). Therefore, the following problems
were left unsolved by \BMnM:
\begin{itemize}
	\item[(i)] prove differential equations satisfied by $Q(0,0;t)$ and $Q(t)=Q(1,1;t)$\cite[\S7.5]{BMM};
	\item[(ii)] find closed form expressions for them~\cite[\S7.3]{BMM};
	\item[(iii)] classify models with algebraic generating function $Q(t)$~\cite[\S7.3]{BMM}.
\end{itemize}

\begin{table}[t]
\centerline{%
\begin{tabular}{|cccc||cccc|}
\hline
     & $\cS$ & occurring $\pFqnoargs21$ & $w$
&    & $\cS$ & occurring $\pFqnoargs21$ & $w$ \\
\hline
   1 & \stepset10101010 & $\pFq21{\frac12,\frac12}{1}w$ & $16t^2$
& 11 & \stepset10011100 & $\pFq21{\frac12,\frac12}{1}w$ & ${\frac{16t^2}{4t^2+1}}$ \\
   2 & \stepset01010101 & $\pFq21{\frac12,\frac12}{1}w$ & $16t^2$
& 12 & \stepset10111110 & $\pFq21{\frac14,\frac34}{1}w$ & ${\frac{64t^3(2t+1)}{(8t^2-1)^2}}$ \\
   3 & \stepset11011101 & $\pFq21{\frac14,\frac34}{1}w$ & ${\frac{64t^2}{(12t^2+1)^2}}$
& 13 & \stepset01011101 & $\pFq21{\frac14,\frac34}{1}w$ & ${\frac{64t^2(t^2+1)}{(16t^2+1)^2}}$ \\
   4 & \stepset11111111 & $\pFq21{\frac12,\frac12}{1}w$ & ${\frac{16t(t+1)}{(4t+1)^2}}$
& 14 & \stepset01111111 & $\pFq21{\frac14,\frac34}{1}w$ & ${\frac{64t^2(t^2+t+1)}{(12t^2+1)^2}}$ \\
   5 & \stepset01001001 & $\pFq21{\frac14,\frac34}{1}w$ & $64t^4$
& 15 & \stepset10010100 & $\pFq21{\frac14,\frac34}{1}w$ & $64t^4$ \\
   6 & \stepset01101011 & $\pFq21{\frac14,\frac34}{1}w$ & ${\frac{64t^3(t+1)}{(1-4t^2)^2}}$
& 16 & \stepset10110110 & $\pFq21{\frac14,\frac34}{1}w$ & ${\frac{64t^3(t+1)}{(1-4t^2)^2}}$ \\
   7 & \stepset11001001 & $\pFq21{\frac12,\frac12}{1}w$ & ${\frac{16t^2}{4t^2+1}}$
& 17 & \stepset10010010 & $\pFq21{\frac13,\frac23}{1}w$ & $27t^3$ \\
   8 & \stepset11101011 & $\pFq21{\frac14,\frac34}{1}w$ & ${\frac{64t^3(2t+1)}{(8t^2-1)^2}}$
& 18 & \stepset10111011 & $\pFq21{\frac13,\frac23}{1}w$ & $27t^2(2t+1)$ \\
   9 & \stepset11010101 & $\pFq21{\frac14,\frac34}{1}w$ & ${\frac{64t^2(t^2+1)}{(16t^2+1)^2}}$
& 19 & \stepset00110011 & $\pFq21{\frac12,\frac12}{1}w$ & $16t^2$ \\
  10 & \stepset11110111 & $\pFq21{\frac14,\frac34}{1}w$ & ${\frac{64t^2(t^2+t+1)}{(12t^2+1)^2}}$
&    & & & \\
\hline
\end{tabular}
}

\caption{Hypergeometric series occurring in explicit expressions for~$Q(x,y;t)$.
The $\pFqnoargs21$ are given up to contiguity and derivation, that is,
up to integer shifts of the parameters.}
\label{tab:2F1-in-Fxyt}
\end{table}

\paragraph{Contributions.}

The original goal of the present paper was to answer question~(i).
In this regard, we rigorously prove the differential equations for $Q(0,0;t)$
and $Q(1,1;t)$ guessed by Bostan and Kauers~\cite{BK,BoKa} for the 19 models
with D-finite transcendental complete generating function. We actually do
more, that is we also find and prove differential equations for $Q(x,0;t)$ and
for $Q(0,y;t)$, that specialize to equations for
$Q(1,0;t)$ and $Q(0,1;t)$.

By solving these differential equations, we answer (ii) in the following
sense: for all the 19 models mentioned above we uniformly find closed form
expressions for $Q(x,y;t)$ in terms of Gauss' hypergeometric
series~$\pFqnoargs21$ with parameters $a,b,c \in \mathbb{Q}$, with $-c \notin \mathbb{N}$,
defined by
\begin{equation}\label{eq:2F1}
\twoFone{a}{b}{c}{t}
 = \sum_{n=0}^\infty \frac{(a)_n(b)_n}{(c)_n} \, \frac
{t^n} {n!},
\end{equation}
where $(x)_n$ denotes the Pochhammer symbol $(x)_n=x(x+1)\cdots(x+n-1)$ for $n\in\mathbb{N}$.

More precisely, we obtain the following structure result, that has been
conjectured in~\cite[\S3.2]{BK}. Note that a similar expression also appears
in a related combinatorial context~\cite{BoChHoPe11} for rook paths on a
three-dimensional chessboard.

\begin{thm} \label{thm:structure}
Let $\cS$ be one of the 19 models of small step walks in the quarter plane (see Table~\ref{tab:steps}). The complete generating function $Q(x,y;t)$ is expressible as a finite sum of iterated integrals of products of algebraic functions in $x,y,t$ and of expressions of the form $\twoFone{a}{b}{c}{w(t)}$, where $c\in\mathbb{N}$ and $w(t) \in \mathbb{Q}(t)$.
\end{thm}
The parameters $a,b,c$ of the occurring $\pFqnoargs21$'s as well as the rational functions $w(t)$ are explicitly given in Table~\ref{tab:2F1-in-Fxyt}.
The full expressions of the generating functions $Q(0,0;t)$, $Q(0,1;t)$, $Q(1,0;t)$, $Q(1,1;t)$,
$Q(x,0;t)$, $Q(0,y;t)$ and $Q(x,y;t)$ are too large to be displayed in this
paper, and are available on-line at
\url{http://specfun.inria.fr/chyzak/ssw/closed_forms.html}.
It turns out by inspection that the involved hypergeometric functions have a
very particular form: they are intimately related to elliptic
integrals,
namely to the complete elliptic integrals of first and second kinds,
\begin{align*}
K(k) = \int_0^{\pi/2} (1-k^2 \sin^2\theta)^{-1/2} \ d\theta &=
  \frac\pi2 \pFq21{\frac12,\frac12}{1}{k^2} , \\
E(k) = \int_0^{\pi/2} (1-k^2 \sin^2\theta)^{1/2} \ d\theta &=
  \frac\pi2 \pFq21{-\frac12,\frac12}{1}{k^2} .
\end{align*}
For instance, for the step set
$\{ (1, 1), (0, 1), (-1, 1), (-1, 0), (-1, -1), (0, -1), (1, -1), (1, 0) \}$
of the so-called \textit{king walks} (case 4 in Table~\ref{tab:steps}),
we prove that
\begin{equation}\label{closed-formula-kings}
Q(t) =
\frac{1}{t} \int_0^t \frac{1}{(1+4x)^3} \cdot \twoFone{\frac32}{\frac32}{2}{\frac{16x(1+x)}{(1+4x)^2}} \, dx .
\end{equation}
See Section~\ref{sec:one-example} for a detailed presentation of this example.
Alternatively, an expression of~$Q(t)$ in terms of elliptic integrals is
\begin{equation*}
Q(t) =
  \frac1t \int_0^t
  \frac1{\pi (1+4x)^2 \sqrt{x(1+x)}} \cdot
  K'\left(\frac{4\sqrt{x(1+x)}}{1+4x}\right) \, dx .
\end{equation*}

The relationship to elliptic integrals appears to hold true
in a far more general setting.
Indeed, taking Theorem~\ref{thm:structure} as starting point,
one of us (van~Hoeij) has checked that
for \emph{many} (more than 100) integer sequences $(a_n)_{n \geq 0}$ in the OEIS
whose generating function $A(t)=\sum_{n\geq 0} a_n t^n$ is
both D-finite and convergent in a small neighborhood of $t=0$,
all second-order irreducible factors
of the minimal-order linear differential operator
annihilating~$A(t)$ are solvable
either in terms of algebraic functions,
or in terms of complete elliptic integrals.
This surprisingly general feature,
reminiscent of Dwork's conjecture mentioned in~\cite[\S3.2]{BK},
begs for a combinatorial explanation.
See also~\cite[Section~8]{HoKu16} for a similar discussion.

In Theorem~\ref{thm:structure}
and in representations of generating functions
like~\eqref{closed-formula-kings},
all ``functions'' bear a combinatorial meaning:
they have to be understood as denoting formal series at~0,
potentially with a (finite) polar part.
Correspondingly, integration has to be viewed as a linear operator
from the set of formal Laurent series without term in~$t^{-1}$
to the whole field~$\set C((t))$ of formal Laurent series.
By the natural growth of the number of walks counted by length,
all series considered can also be viewed as analytic series
that converge at least on an annulus around~0.
This alternative interpretation will be used
in Section~\ref{sec:asymptotics} only,
for asymptotic considerations.

Finally, concerning question (iii), we start from the explicit differential
equations and we exhaustively classify the algebraic cases among all the
specializations of the generating function~$Q(x,y;t)$ at points $(x,y) \in
\{(0,0),(1,0),(0,1),(1,1)\}$. As a corollary, we reprove the transcendence
of~$Q(x,y;t)$. More precisely, we prove:

\begin{thm} \label{thm:trans}
Let $\cS$ be one of the 19 models of small step walks in the quarter plane (see Table~\ref{tab:steps}), with complete generating function $Q(x,y;t)$.
For any $(\alpha,\beta) \in \{(0,0),(1,0),(0,1),(1,1)\}$, the power series $Q(\alpha,\beta;t)$ is transcendental, except in the following four cases:
\begin{itemize}
  \item Case 18 at $(\alpha,\beta)=(1,0)$ and at $(\alpha,\beta)=(0,1)$,
  \item Cases 17 and 18 at $(\alpha,\beta)=(1,1)$.
\end{itemize}
As a consequence, the power series $Q(x,y;t)$, $Q(x,0;t)$, and~$Q(0,y;t)$ are
transcendental for all the 19~models.
Additionally, the generating functions of the four algebraic cases are equal to:
\begin{itemize}
  \item
	$Q(1,1) = \frac{1}{2t^2} \left(1-t-\sqrt{(1+t)(1-3t)}\right)$ \ in case 17,
  \item
	$Q(1,1) = \frac{1}{8t^2} \left(1-2t-\sqrt{(1+2t)(1-6t)}\right)$ \ in case 18,
  \item
	$Q(1,0) = Q(0,1) = \frac{1}{32t^3} \left( (1-6t)^{3/2} (1+2t)^{1/2}-4t^2+8t-1 \right)$ \ in case 18.
\end{itemize}
\end{thm}

As an aside, starting from the explicit expressions in terms of hypergeometric
functions, we use singularity analysis and transfer theorems that are
classical in Analytic Combinatorics to get some asymptotic formulas for the
$n$th coefficient of $Q(0,0;t)$, $Q(1,0;t)$, $Q(0,1;t)$ and $Q(1,1;t)$.

\paragraph{Methodology.} Our proofs of Theorems~\ref{thm:structure}
and~\ref{thm:trans} are computer-driven and crucially rely on the use of
several modern computer algebra algorithms. The starting point is a result by
\BMnM~\cite{BMM}, stating that for the 19 models in Table~\ref{tab:steps} the
complete generating function $Q(x,y;t)$ can be expressed as the \emph{positive
part} of a certain rational function in three variables. The notion of
positive part is one of the key mathematical ingredients in what follows.
In one variable, extraction of the positive part is an operator,
denoted~$[x^>]$, which acts on formal Laurent series by cutting away all the
terms with zero or negative exponents, leaving a formal power series with no
constant term as a result. For example,
\[
  [x^>]\frac1{x^2(1-x)^2} = [x^>] \sum_{n=-2}^\infty(n+3)x^n
   = \sum_{n=1}^\infty (n+3)x^n = \frac{x(4-3x)}{(1-x)^2} .
\]
Note that interpreting the rational function as a formal Laurent series
in~$x^{-1}$ instead of~$x$ would lead to a different extraction map.  Indeed, for this
other definition of positive-part extraction, we would have
\[
 [x^>]\frac1{x^2(1-x)^2} = [x^>] \sum_{n=0}^\infty n\frac1{x^{n+3}} = [x^>] \sum_{n=-\infty}^{-3} -(n+3) x^n = 0.
\]
Things get more complicated in the multivariate setting. Using the kernel
method, \BMnM\ showed in~\cite[Prop.~8]{BMM} that the generating function
$Q(x,y;t)$, can be written in the form
\begin{equation}\label{eq:Q-as-pos-part}
  Q(x,y;t) = \frac1{x y} [x^>][y^>] \frac{N(x,y)}{1-tS(x,y)}
\end{equation}
where $N(x,y)$ and $S(x,y)$ are certain (structured) Laurent polynomials
in~$y$ with coefficients that are rational functions in~$x$. These quantities
depend on $\cS$ and are listed in Table~\ref{tab:steps}. Since there is no
unique natural way of mapping rational functions in several variables to
multivariate formal Laurent series, it is a priori not clear how the
positive-part extraction is defined in this context. Here is the intended
reading of~\eqref{eq:Q-as-pos-part}: first interpret $N(x,y)/(1-tS(x,y))$ as
an element of $\mathbb{Q}(x)[y,1/y][[t]]$, owing to particular properties of $N$
and~$S$ (see Lemma~\ref{lem:same-pos-parts} below); let $[y^>]$ act term by
term, obtaining a series in $\mathbb{Q}(x)[y][[t]]$ that can be shown to
actually belong to $\mathbb{Q}[x,1/x][y][[t]]$ for all cases in
Table~\ref{tab:steps}; then let $[x^>]$ act term by term, finally obtaining an
element of $\mathbb{Q}[x][y][[t]]$. In this reading, the composition
$[x^>][y^>]$ of positive-part operators is only applied to Laurent
polynomials, for which it is certainly well-defined, in a unique way.

As pointed out by \BMnM, Equation~\eqref{eq:Q-as-pos-part} already implies the
D-finiteness of $Q(x,y;t)$, since positive parts can be encoded as diagonals,
and diagonals of D-finite power series are again D-finite~\cite{Lipshitz88}.
This argument also implies an algorithm for computing linear
differential equations satisfied by $Q(x,y;t)$, since the D-finiteness proof
in~\cite{Lipshitz88} is effective and basically amounts to linear algebra.
Therefore, from~\eqref{eq:Q-as-pos-part} one could, in principle, determine
differential equations for~$Q(x,y;t)$.
To be more specific,
the positive part of a formal power series~$R \in \set Q[[x,y,t]]$
can be encoded as
\begin{equation}\label{eq:pp-as-hp-as-diag}
[x^>][y^>] R(x,y;t) =
  \frac{x}{1-x}\frac{y}{1-y} \odot_{x,y} R(x,y;t) =
  \Diag_{x,x'} \Diag_{y,y'} \frac{x}{1-x}\frac{y}{1-y} R(x',y';t) ,
\end{equation}
where the Hadamard product denoted~$\odot_{x,y}$ is
the term-wise product of two series,
while the diagonal operator~$\Diag_{x,x'}$ selects those terms
with equal exponents of $x$ and~$x'$.
However, the direct use of~\eqref{eq:pp-as-hp-as-diag} in our context
leads to infeasible computations;
worse, the intermediate algebraic objects involved in the calculations
would probably have too large sizes to be merely written and stored.
This is really unfortunate, since our need is
mere evaluations of the diagonals in~\eqref{eq:pp-as-hp-as-diag}
at specific values for $x$ and~$y$.

To bypass this computational obstacle, we use two ingredients.

The first one
is our main theoretical innovation: we reformulate the generating function
$Q(x,y;t)$ in terms of \emph{formal residues}.
This idea is classical
(and in fact already used in Lipshitz' proof~\cite{Lipshitz88}):
we encode diagonals as residues,
with the added advantage that early specialization
of the variables $x$ and/or~$y$ becomes possible.
Additionally, our derivation bases on a positive-part extraction
that differs from \BMnM's iterated operator~$[x^>][y^>]$:
we use a theory~\cite{AparicioMonforteKauers-2013-FLS} of series
with exponents that may be arbitrarily large in negative directions,
but are restricted to fixed cones,
together with a different, direct positive-part operator, $[x^>y^>]$, to be
defined in Section~\ref{sec:positiveParts}.
The outcome of this
reformulation is the ability to compute linear differential equations with
polynomial coefficients for the specializations $Q(x,0;t)$ and~$Q(0,y;t)$.

To perform these computations, we use a second ingredient,
\emph{creative telescoping}, an efficient algorithmic technique
for the symbolic integration of multivariate functions.
Indeed, a direct application of Lipshitz' linear-algebra algorithm
(even with specialized variables)
still leads to too large systems,
while creative telescoping succeeds in our cases of application;
see~\cite[\S2.3]{BoChHoPe11} for a related discussion.
By specialization and recombination,
the equations thus obtained
give rise to rigorously proved differential equations for
$Q(0,0;t)$, $Q(1,0;t)$, $Q(0,1;t)$ and $Q(1,1;t)$, thus answering
question~(i). The analysis of these differential equations combined with
Kovacic's algorithm~\cite{Kovacic-1986-ASS} allows to answer question~(iii)
and to prove Theorem~\ref{thm:trans}. Moreover, these differential equations
are solved in explicit terms using symbolic algorithms for ODE factorization
and ODE solving, leading to the proof of Theorem~\ref{thm:structure} and to
the answer of question~(ii).

The remarkable property that the differential equations in the $19 \times 4$
cases
could \emph{all} be solved in terms of hypergeometric functions relies on the
fact that these operators share a very peculiar factorization pattern: they
factor into factors that all have order 1 with the exception of the left-most
one that can have order 1 or 2.
The origin of this common mathematical feature deserves to be better understood.

\paragraph{Previous work.} We conclude this introduction by mentioning
previous contributions on the main topics of the article: D-finiteness,
transcendence and explicit expressions for $Q(x,y;t)$ and its specializations.

\medskip \noindent \emph{D-finiteness}. For the simplest models, the
\emph{square walk} (case 1) and the \emph{diagonal walk} (case 2),
D-finiteness is classical~(e.g.,~\cite{Arques86}). Some more involved models
have been considered sporadically: for case 19,
Gouyou-Beauchamps~\cite{gouyou86} proved bijectively that $Q(1,0;t)$ is
D-finite; for cases 5 and 15 Mishna~\cite[\S2.4.1, \S2.4.2]{Mishna09} showed
that $Q(x,y;t)$ is D-finite using the kernel method; and for case 17 she
proved~\cite[\S2.3.3]{Mishna09} that $Q(x,y;t)$ is D-finite by using a
bijection with Young tableaux of height at most~3.

Several methods have been proposed to capture D-finiteness in a uniform way.
For models with a vertical symmetry (cases 1--16), Bousquet-M\'elou and
Petkov{\v{s}}ek~\cite[\S2]{BoPe03} proved that $Q(t) = Q(1,1;t)$ is D-finite
by a combinatorial argument, and Bousquet-M\'elou~\cite[\S3]{mbm02} proved
D-finiteness of $Q(x,y;t)$ by an algebraic argument (a variation of the
{kernel method}). For cases 17--19, Gessel and Zeilberger~\cite{GeZe92} proved
D-finiteness of $Q(x,y;t)$ by using an algebraic version of the reflection
principle; their argument works more generally when the step set is left
invariant by a Weyl group and the walks are confined to a corresponding Weyl
chamber. Bousquet-M\'elou and Mishna~\cite{BMM} reproved D-finiteness of
$Q(x,y;t)$ in all 19 cases by using the kernel equation and the group of the
walk borrowed from~\cite{FaIaMa99}; their method generalizes the previous ones
from~\cite{GeZe92} and~\cite{mbm02}. Raschel~\cite{Raschel12} uses boundary
value problems to get integral representations for~$Q(x,y;t)$ that imply its
D-finiteness in all 19 cases.
By using methods from the book~\cite{FaIaMa99},
Fayolle and Raschel reproved in~\cite[Theorem~1.1]{FaRa10}
that $(x, y) \mapsto Q(x, y; t_0)$ is D-finite for each
$t_0\in(0,\#S^{-1}]$ and in all 19~cases.

\medskip \noindent \emph{Transcendence}. Algebraicity/transcendence proofs
were first considered in some isolated cases: in case~15, $Q(x,y;t)$ was
proved transcendental by Mishna~\cite[Th.~2.5]{Mishna09}; in case~17,
Mishna~\cite[\S2.3.3]{Mishna09}, then Bousquet-M\'elou and
Mishna~\cite[\S5.2]{BMM}, showed that $Q(x,y;t)$ and $Q(0,0;t)$ are
transcendental and that $Q(1,1;t)$~is algebraic;
in case~18, $Q(1,1;t)$~was proved
algebraic by Bousquet-M\'elou and Mishna~\cite[\S5.2]{BMM}; in case 19,
Bousquet-M\'elou and Mishna~\cite[\S5.3]{BMM} showed that $Q(0,0;t)$,
$Q(0,1;t)$, $Q(1,0;t)$ and $Q(1,1;t)$ are transcendental. The first unified
transcendence proof for $Q(x,y;t)$ applying to all 19 cases is by Fayolle and
Raschel~\cite[Theorem 1.1]{FaRa10}, although they attribute that result to
Bousquet-Mélou and Mishna~\cite{BMM}. They actually proved more, namely that
$Q(x, y; t_0)$ is transcendental for each~$t_0\in(0,\#S^{-1}]$, using the
approach in~\cite[Chap.~4]{FaIaMa99}.
However, this result does not provide any transcendence information
about specializations at $x,y \in \{0,1\}$.

\medskip \noindent \emph{Explicit equations and formulas}. For the simplest
models (cases 1--2), simple formulas exist, see e.g.~\cite{mbm02}; some of
them admit bijective proofs, see e.g.~\cite{GKS92}. For models 5 and 15,
Mishna~\cite[Th. 2.5 and 2.6]{Mishna09} gives explicit expressions of
$Q(x,y;t)$ in terms of some auxiliary series. For model 17, basing on earlier
work by Regev~\cite{Regev81} and using a bijection with Young tableaux of
height at most~3, Mishna~\cite[\S2.3.3]{Mishna09} shows that $q_{i,j;n}$ has a
nice hypergeometric expression and that $Q(1,1;t)$ is the (algebraic)
generating function of Motzkin numbers. In cases 17--18, Bousquet-M\'elou and
Mishna~\cite[\S5.2]{BMM} gave explicit expressions for $Q(1,1;t)$ (and more
generally for $Q(x,1/x;t)$). For model 19, it was proved by
Gouyou-Beauchamps~\cite{gouyou86} that the number of $n$-step walks ending on
the $x$-axis is a product of Catalan numbers; hypergeometric expressions for
the total number of walks are derived in~\cite[\S5.3]{BMM}, and for those
ending at an arbitrary point $(i,j)$. Bostan and Kauers~\cite{BK,BoKa}
empirically determined differential equations for $Q(1,1;t)$ and $Q(0,0;t)$
for all 19 models. For the ``highly symmetric" models (cases 1--4), Melczer
and Mishna~\cite{MeMi14} proved differential equations for $Q(1,1;t)$
conjectured by Bostan and Kauers~\cite{BK}.
Raschel~\cite{Raschel12} provides explicit integral representations of the
complete generating function using a uniform analytic approach.

\medskip Before we enter into the details, Section~\ref{sec:one-example} goes
through the whole process on one concrete example. From now on, we more simply
write $\bar x$ and~$\bar y$, respectively, for $x^{-1}$ and~$y^{-1}$.

\section{A Worked Example: King Walks in the Quarter Plane} \label{sec:one-example}

We illustrate our approach on an example. We choose the so-called \textit{king
walks} in the quarter plane (case~4 in Table~\ref{tab:steps}), with step set
$\cS = \{ (1, 1), (0, 1), (-1, 1), (-1, 0), (-1, -1), (0, -1), (1, -1), (1, 0) \}$.
The first terms of the generating function $Q(1,1;t)$ of king walks with
prescribed length and arbitrary endpoint read (see \url{http://oeis.org/A151331})
\[Q(1,1;t) = 1+3t+18t^2+105t^3+684t^4+4550t^5+31340t^6+219555t^7+1564080t^8+ \cdots,
\]
and the methods of the present article allow to obtain
the above-mentioned closed formula~\eqref{closed-formula-kings} for it.

Here are the main steps of our approach. First, the classical kernel equation~\cite[Lemma~4]{BMM} relates~$Q(x,y;t)$ to $Q(x,0;t)$, $Q(0,y;t)$, and~$Q(0,0;t)$:
\begin{equation}\label{kernel-eq-kings}
  Q(x,y;t) = 1 + t \Big( S(x,y) Q(x,y;t)
  -\bar{y}(x+1+\bar{x}) Q(x,0;t)
  -\bar{x}(y+1+\bar{y}) Q(0,y;t)
  +\bar{x}\bar{y} Q(0,0;t)\Big) ,
\end{equation}
where $S(x,y)$~is the generating polynomial of the step set:
\begin{equation*}
S(x,y) = \sum_{(i,j)\in\cS}x^iy^j = xy+y+\bar{x}y+x+\bar{x}+x\bar{y}+\bar{y}+\bar{x}\bar{y} .
\end{equation*}
A simple but important observation is that the kernel $K(x,y;t) = 1 -
tS(x,y)$ remains unchanged under the change of variables $(x,y) \leftarrow
(x,\bar{y})$, $(x,y) \leftarrow (\bar{x},{y})$ and $(x,y) \leftarrow
(\bar{x},\bar{y})$.  (Other step sets could require different changes of
variables to provide the same property.)

Applying these rational transformations to the kernel equation~\eqref{kernel-eq-kings} yields the four relations:
\begin{align*}
xyK(x,y;t)Q(x,y;t) &= \phantom{-} xy - tx(x+1+\bar{x})Q(x,0;t) - ty(y+1+\bar{y})Q(0,y;t) + tQ(0,0;t), \\
-\bar{x}yK(x,y;t)Q(\bar{x},y;t) &= - \bar{x}y + t\bar{x}(x+1+\bar{x})Q(\bar{x},0;t) + ty(y+1+\bar{y})Q(0,y;t) - tQ(0,0;t), \\
\bar{x}\bar{y}K(x,y;t)Q(\bar{x},\bar{y};t) &= \phantom{-} \bar{x}\bar{y} - t\bar{x}(x+1+\bar{x})Q(\bar{x},0;t) - t\bar{y}(y+1+\bar{y})Q(0,\bar{y};t) + tQ(0,0;t),\\
-x\bar{y}K(x,y;t)Q(x,\bar{y};t) &= - x\bar{y} + tx(x+1+\bar{x})Q(x,0;t) + t\bar{y}(y+1+\bar{y})Q(0,\bar{y};t) - tQ(0,0;t).
\end{align*}
Upon adding up these equations, all terms in the right-hand side
involving~$Q$ disappear, resulting in
\[
xyQ(x,y;t)-\bar{x}yQ(\bar{x},y;t)+\bar{x}\bar{y}Q(\bar{x},\bar{y};t)-x\bar{y}Q(x,\bar{y};t)=
{K(x,y;t)}^{-1} \left(xy-\bar{x}y+\bar{x}\bar{y}-x\bar{y}\right).
\]
Now, the main observation is that on the left-hand side, all terms except the
first one involve negative powers either of $x$ or of $y$. Therefore, extracting
positive parts expresses the generating series $xy Q(x,y;t)$ as the positive
part (w.r.t.~$x$ and $y$) of a trivariate rational function:
\begin{equation} \label{pospart-kings}
	xy Q(x,y;t) = [x^>][y^>] \left( \frac{xy-\bar{x}y+\bar{x}\bar{y}-x\bar{y}}{1-t(xy+y+y\bar{x} + \bar{x}+\bar{x}\bar{y}+\bar{y} +x\bar{y}+x)}\right).
\end{equation}

\begin{comment}
The first terms of this positive part read:
\begin{align*}
x y+(x^2 y+x y^2+x^2 y^2) t+(3 x y+2 x^2 y+2 x^3 y+2 x y^2+2 x^2 y^2+2 x^3 \\ y^2+2 x y^3+2 x^2 y^3+x^3 y^3) t^2+(6 x^3 y+4 x^4 y+6 x y+11 x^2 y+11 x y^2+16 \\
 x^2 y^2+9 x^3 y^2+5 x^4 y^2+6 x y^3+3 x^4 y^3+9 x^2 y^3+6 x^3 y^3+4 x y^4+3 \\ x^3 y^4+5 x^2 y^4+x^4 y^4) t^3+O(t^4)
\end{align*}
which is in accordance with the known expansion of $xy Q(x,y;t)$.
\end{comment}

Up to this point, our reasoning is borrowed from Bousquet-M\'elou's and
Mishna's article~\cite{BMM}. Combined with Lipshitz's
result~\cite{Lipshitz88} that positive parts of D-finite functions are
D-finite, it already implies that $Q(x,y;t)$ is D-finite; in particular,
$Q(1,1;t)$ is also D-finite.  Our aim is to refine this qualitative result,
and explicitly obtain a linear differential equation satisfied by
$Q(1,1;t)$. Such a differential equation was algorithmically
\emph{guessed\/} in~\cite{BK} starting from the first terms in the power
series expansion of $Q(1,1;t)$; however the methods in~\cite{BK} do not
provide a \emph{rigorous proof\/} of the correctness of that equation.

Starting from~\eqref{pospart-kings} and following more closely Lipshitz'
encoding~\cite{Lipshitz88}, a first observation is that
$Q(x,y;t)$ is equal to the iterated diagonal
\begin{equation} \label{diag-encoding}
	 \Diag_{x_1,x_2} \Diag_{y_1,y_2} \left(\frac{x_2 y_2 (x_1y_1-\bar{x}_1y_1+\bar{x}_1\bar{y}_1-x_1\bar{y}_1)}{(1-x_2)(1-y_2)(1-t(x_1y_1+y_1+y_1\bar{x}_1 + \bar{x}_1+\bar{x}_1\bar{y}_1+\bar{y}_1 +x_1\bar{y}_1+x_1))} \right).
\end{equation}
There exist algorithms that take as input a rational function and compute a
system of partial differential equations satisfied by its
diagonal~\cite{Zeilberger90,Chyzak00,Koutschan10,Lairez16}.
However, in our case, these computations are too
difficult, and exceed by far the limits of the best existing algorithms.
The reason is that differential equations w.r.t.~$t$ and with polynomial
coefficients in $x,y,t$ are really huge, so the main limitation of algorithms
computing~\eqref{diag-encoding} already comes from the size of the output.
Another weakness of the diagonal encoding~\eqref{diag-encoding} is that it
does not provide direct access to the univariate series $Q(1,1;t)$, since
taking diagonals and specializing variables are operations that do not
commute.

To circumvent these difficulties and to make the computation feasible, our
key idea is to encode the positive part in~\eqref{pospart-kings} as a
formal residue:
\begin{equation}\label{residue-encoding-kings}
Q(\alpha,\beta;t) = \Res_{x,y} \left( \frac{xy-\bar{x}y+\bar{x}\bar{y}-x\bar{y}}{(1-\alpha x)(1-\beta y)(1-t(xy+y+y\bar{x} + \bar{x}+\bar{x}\bar{y}+\bar{y} +x\bar{y}+x))}\right).
\end{equation}
The formal proof of this encoding is delicate and is the topic of
Section~\ref{sec:positiveParts}.  The advantage of~\eqref{residue-encoding-kings} over~\eqref{diag-encoding}
is twofold. On the one hand, the residue computation can be carried out by using
a single call to the creative-telescoping algorithm for rational functions,
while the diagonal computation~\eqref{diag-encoding} has two steps, the first
for a rational function in five variables, the second for an algebraic function
in four variables. On the other hand, and more importantly, taking residues
commutes with specialization, contrarily to positive parts and diagonals.
Therefore, the generating series for walks $Q(1,1;t)$ is equal to
\begin{equation*}
Q(1,1;t) = \Res_{x,y} \left( \frac{xy-\bar{x}y+\bar{x}\bar{y}-x\bar{y}}{(1-x)(1-y)(1-t(xy+y+y\bar{x} + \bar{x}+\bar{x}\bar{y}+\bar{y} +x\bar{y}+x))}\right).
\end{equation*}

A nice property of residues is that pure derivatives in $x$ and~$y$ have zero
residue. In order to compute a linear differential equation satisfied by
$Q(1,1;t)$ it thus suffices to find $U,V \in \set Q(x,y;t)$ and a differential
operator $L$ in $\set Q(t) \langle \partial_t \rangle$ such that
\begin{equation}\label{eq:P-as-derivatives-Kings}
L \left(
\frac{xy-\bar{x}y+\bar{x}\bar{y}-x\bar{y}}{(1-x)(1-y)(1-t(xy+y+y\bar{x} +
\bar{x}+\bar{x}\bar{y}+\bar{y} +x\bar{y}+x))}\right) = \partial_x (U) +
\partial_y (V) ,
\end{equation}
where $\partial_t$, $\partial_x$, and~$\partial_y$ stand
for the derivation operators $d/dt$, $d/dx$, and~$d/dy$, respectively.
Indeed, since $L$ depends neither on~$x$ nor on~$y$, it commutes with
the extraction of the residue with respect to $x$ and~$y$. The fact that
the residue of $\partial_x (U) + \partial_y (V)$ is zero then proves that
$L(Q(1,1;t))=0$.

However, this last implication is more subtle than it looks:
Equation~\eqref{eq:P-as-derivatives-Kings} really is a relation
between rational functions;
the conclusion about~$Q(t)$ is a conclusion about series.
Owing to the ambiguity inherent to series expansions of rational functions,
it is not clear a priori that
the rational function in the left-hand side of~\eqref{eq:P-as-derivatives-Kings}
and $U$ and~$V$ in its right-hand side
can be expanded consistently
and in accordance with the combinatorial interpretation for~$Q(t)$.
Describing the right expansion is the topic of Section~\ref{sec:positiveParts}:
Theorem~\ref{thm:P-by-CT} provides the wanted implication.

Now, the existence of a non trivial $(L,U,V)$
satisfying~\eqref{eq:P-as-derivatives-Kings}
is guaranteed by the
theory of holonomic D-modules, which was first used in this context by
Zeilberger~\cite{Zeilberger90b} (see also~\cite{Cartier92}). The explicit
form of a solution $(L,U,V)$ can be determined using \emph{creative
  telescoping}~\cite{Zeilberger90,Chyzak00,Koutschan10};
a detailed description of the algorithm
is presented in~\cite[\S2.4.2]{BoChHoPe11}.

In our case, we find
\begin{multline}\label{eq:P}
  L = t^2 (1+4 t) (8 t-1) (2 t-1) (1+t) \partial_t^3+t (200 t^3+576 t^4-33 t-252 t^2+5) \partial_t^2 \\
     +4 (22 t^3-117 t^2-12 t+288 t^4+1) \partial_t +384 t^3-12-144 t-72 t^2 .
\end{multline}
Note that
this is precisely the differential operator \emph{guessed\/} in~\cite{BK}.

Moreover, factorization algorithms for linear differential
operators~\cite{Hoeij97} can be used to prove that $L = L_2 L_1$, where
$L_1 = \partial_t + 1/t$ and
\begin{multline}\label{eq:P2}
L_2 =
t^2 (1+4 t) (1 - 8 t) (1-2 t) (1+t) \partial_t^2+2 t (256 t^4+80 t^3-111 t^2-14 t+2) \partial_t\\+768 t^4+8 t^3-306 t^2-30 t+2.
\end{multline}
It follows that the Laurent power series
\begin{equation*}
f(t) = \frac{dQ}{dt}(1,1;t) +
\frac{Q(1,1;t)}{t} = t^{-1}+6+54t+420t^2+3420t^3+27300t^4+219380t^5+O(t^6)
\end{equation*}
is a solution of $L_2$.  Starting from the second order operator $L_2$,
algorithmic methods explained in~\cite[\S2.6]{BoChHoPe11} (see
also~\cite{FaHo11}) allow to express $f(t)$ in terms of a
$_2F_1$ hypergeometric function:
\[f(t) = \frac{1}{t(1+4t)^3} \cdot  \twoFone{\frac32}{\frac32}{2}{\frac{16t(1+t)}{(1+4t)^2}}.\]

Finally, solving the equation ${d}/{dt} \, Q(1,1;t) + Q(1,1;t)/t = f(t)$ yields formula~\eqref{closed-formula-kings}.

Similarly, by deriving variants of~\eqref{eq:P-as-derivatives-Kings}
parametrized by \emph{indeterminate\/} $\alpha$ and~$\beta$ we obtain the
representations
\begin{align*}
Q(\alpha,0;t)
&= \Res_{x,y} \left( \frac{xy-\bar{x}y+\bar{x}\bar{y}-x\bar{y}}{(1-\alpha x)(1-t(xy+y+y\bar{x} + \bar{x}+\bar{x}\bar{y}+\bar{y} +x\bar{y}+x))}\right) \\
\intertext{and}
Q(0,\beta;t)
&= \Res_{x,y} \left(
\frac{xy-\bar{x}y+\bar{x}\bar{y}-x\bar{y}}{(1-\beta y)(1-t(xy+y+y\bar{x} +
  \bar{x}+\bar{x}\bar{y}+\bar{y} +x\bar{y}+x))}\right) .
\end{align*}
Creative-telescoping techniques still allow the effective computation of
differential operators in $\set Q(\alpha;t)\langle\partial_t\rangle$ for $Q(\alpha,0;t)$,
resp.~in $\set Q(\beta;t)\langle\partial_t\rangle$ for $Q(0,\beta;t)$.  Owing to the
additional symbolic indeterminate, the computations are much harder than
for $Q(1,1;t)$, but still feasible. Each of the resulting differential
operators factors again, this time as a product of an order-two operator
and of three order-one operators. Moreover, the second-order operators are
again solvable in terms of $_2F_1$ functions. Finally, a closed formula for
$Q(\alpha,\beta;t)$ is obtained from the closed formulas for $Q(\alpha,0;t)$ and $Q(0,\beta;t)$
via the kernel equation~\eqref{kernel-eq-kings}. This detour is
computationally crucial, since performing creative telescoping directly on
the five-variable rational function from~\eqref{residue-encoding-kings} is
not feasible even using today's best algorithms.

\section{Computing Positive Parts as Formal Residues}\label{sec:positiveParts}

We will now put the assertions made in the previous section for one
particular case on solid algebraic grounds by clarifying to which series
domains we map our rational functions, by introducing formal notions of
residue, Hadamard product, and positive part for the objects in these
domains, and by showing that the differential equations we obtain from
creative telescoping indeed annihilate the positive parts.

To this end, we study vector spaces of “series” that are just bilateral
infinite arrays in Section~\ref{sec:inf-arrays}, before elaborating
in Section~\ref{sec:series-wrt-cone} on a theory of rings and fields of series
with exponents in cones~\cite{AparicioMonforteKauers-2013-FLS}.  We are
then able to represent positive parts as residues of suitable products in
those rings in Section~\ref{sec:pp-as-res}.  In Section~\ref{sec:spec-pp}, a notion of
cones in opposition is introduced to refine this representation and allow
the early specialization we need for efficiency in computations.  The
connection to our application is done in Section~\ref{sec:gf-as-res}, before we
justify the use of creative telescoping in Section~\ref{sec:ann-by-ct}.

As the present section is of a more general nature than the application
to small-step walks,
we develop the theory here over a general field~$\bK$ of characteristic~0.

\subsection{Linear operations on infinite arrays}
\label{sec:inf-arrays}

The set~$\bK^{\set Z^n}$ consisting of
all the series
\begin{equation}\label{eq:gen-series}
\sum_{i_1,\dots,i_n\in\set Z} a_{i_1,\dots,i_n}x_1^{i_1}\cdots x_n^{i_n}
\end{equation}
is a $\bK$-vector space.  Here, we use the word “series” and the
notation~$\sum$ just for convenience, disregarding any question of
convergence.  The set~$\bK^{\set Z^n}$ is not a ring, but we can endow it
with well-defined operations of residue (from~$\bK^{\set Z^n}$ to~$\bK$) and
positive part (from~$\bK^{\set Z^n}$ to $\bK^{\set N^n}$): for a given series~$f$, the
residue $\Res_{x_1,\dots,x_n} f$ and positive part $[x_1^>\cdots x_n^>] f$
are defined in the natural way as
\begin{equation*}
a_{-1,\dots,-1}
\qquad\text{and}\qquad
\sum_{i_1,\dots,i_n>0} a_{i_1,\dots,i_n}x_1^{i_1}\cdots x_n^{i_n},
\end{equation*}
respectively.  More generally, we may consider residue and positive-part
extractions with respect to some of the variables only, and this will be
obvious from the notation.  This way, positive-part extractions compose
nicely, e.g., $[x_1^>] [x_2^>] = [x_2^>] [x_1^>] = [x_1^> x_2^>]$.  The
vector space~$\bK^{\set Z^n}$ also possesses an internal law of Hadamard
product, defined for two series $f$ and~$g$ with respective coefficients
$a_{i_1,\dots,i_n}$ and~$b_{i_1,\dots,i_n}$ by
\begin{equation*}
f \odot g = \sum_{i_1,\dots,i_n\in\set Z}
  a_{i_1,\dots,i_n} b_{i_1,\dots,i_n} x_1^{i_1}\cdots x_n^{i_n} .
\end{equation*}
Finally, for~$k\leq n$, we will need more operations for series in
different variable sets:
\begin{itemize}
\item for $f \in \bK^{\set Z^k}$ and $g \in \bK^{\set Z^n}$, we define
\begin{equation}\label{eq:semi-hadam-k-n}
f \odotkn g =
  \sum_{i_1,\dots,i_n\in\set Z}
    a_{i_1,\dots,i_k} b_{i_1,\dots,i_n} x_1^{i_1}\cdots x_n^{i_n} ;
\end{equation}
\item for $f \in \bK^{\set Z^n}$ and $g \in \bK^{\set Z^k}$, we define
\begin{equation}\label{eq:semi-hadam-n-k}
f \odotnk g =
  \sum_{i_1,\dots,i_n\in\set Z}
    a_{i_1,\dots,i_n} b_{i_1,\dots,i_k} x_1^{i_1}\cdots x_n^{i_n} .
\end{equation}
\end{itemize}
Beware that, despite the natural embedding of~$\bK^{\set Z^k}$ into~$\bK^{\set
  Z^n}$ mapping $f \in \bK^{\set Z^k}$ to~$fx_{k+1}^0\cdots x_n^0 \in \bK^{\set
  Z^n}$, $f \odot_{k\leq n} g$~in~\eqref{eq:semi-hadam-k-n} is generally
not equal to~$(f x_{k+1}^0\cdots x_n^0)\odot g$ and similarly $f
\odot_{n\geq k} g$~in \eqref{eq:semi-hadam-n-k} is generally not equal
to~$f\odot(g x_{k+1}^0\cdots x_n^0)$.

The vector space~$\bK^{\set Z^n}$ also has a (linear-)differential structure:
it can be endowed with a derivation operator~$\partial_1$ that maps a
series~\eqref{eq:gen-series} to
\begin{equation*}
\sum_{i_1,\dots,i_n\in\set Z} i_1 a_{i_1,\dots,i_n}x_1^{i_1-1}x_2^{i_2}\cdots x_n^{i_n} ,
\end{equation*}
and with operators $\partial_2,\dots,\partial_n$ defined similarly.  The
following result is then obvious.

\begin{lem}\label{lem:res-partial}
For any series~$f \in \bK^{\set Z^n}$ and any~$j$ satisfying $1 \leq j \leq
n$, $\Res_{x_1,\dots,x_n} \partial_j f = 0$.
\end{lem}

Additionally, as Laurent polynomials are finite sums, $\bK^{\set Z^n}$~admits
well-defined operations of multiplication by a Laurent polynomial.
Together with closure under derivation, this makes it meaningful to
consider differential equations in series of~$\bK^{\set Z^n}$.

\subsection{Series expansions with respect to a cone}
\label{sec:series-wrt-cone}

In what follows, we introduce a family of rings of series with support
restricted to cones of~$\set Z^n$.  This extends to rings of series whose
support enjoys a certain compatibility condition with a fixed monomial
ordering.  This discussion is mostly borrowed
from~\cite{AparicioMonforteKauers-2013-FLS}.  See the introduction in that
article for the comparison to the theory of Hahn series (later developed by
Mal'cev, Neumann, and recently Xin) and to the theory of MacDonald series
(recently developed by Aroca, Cano, and Jung).

In order to highlight the names and number of indeterminates in the
notation for various rings and fields of polynomials, series, etc, in what
follows, we will compactly write $U_m$ to denote $m$~indeterminates
$u_1,\dots,u_m$ in notations like $\bK[U_m]$, $\bK(U_m)$, $\bK[[U_m]]$, etc.
To refer to only the last $m-p$ indeterminates $u_{p+1},\dots,u_m$ of~$U_m$,
we will write~$U_{m\setminus p}$.  We will use similar notation for
inverses and combinations of names of indeterminates, e.g.,
$\bK[X_n,Y_{n\setminus k},Y_k^{-1}]$ denotes
$\bK[x_1,\dots,x_n,y_{k+1},\dots,y_n,y_1^{-1},\dots,y_k^{-1}]$.

\begin{definition}
A topologically closed set $C\subseteq\set R^n$ is called a \emph{cone\/}
if for all $u,v\in C$ and $\lambda,\mu\geq0$, $\lambda u+\mu v\in C$.  A
cone~$C$ is called \emph{line-free\/} if $u,-u\in C$ implies $u=0$.
\end{definition}

If $C$ is a line-free cone, then the set consisting of all the series of
the form~\eqref{eq:gen-series} with $a_{i_1,\dots,i_n}=0$ whenever
$(i_1,\dots,i_n)\not\in C$ forms a ring together with the natural addition
and multiplication.  We denote it by $\bK_C[[x_1,\dots,x_n]]$ or~$\bK_C[[X_n]]$.
It is a $\bK$-algebra.  For the trivial cone~$C=\{0\}$, this is simply~$\bK$.
The set $\bK_C\langle\langle X_n\rangle\rangle :=
\bigcup_{(i_1,\dots,i_n)\in\set Z^n} x_1^{i_1}\cdots x_n^{i_n} \bK_C[[X_n]]$,
forms another ring, in
fact just $\bK_C[[X_n]][X_n,X_n^{-1}]$.  Its elements are series whose
supports are contained in a finite union of translated copies
$(i_1,\dots,i_n)+C$ of~$C$. For the trivial cone~$C=\{0\}$, this is simply
the ring of Laurent polynomials, $\bK[X_n,X_n^{-1}]$.  The ring
$\bK_C\langle\langle X_n\rangle\rangle$ inherits the operations of residue
and positive part, the internal law of Hadamard product, and the derivation
operations.  Observe as well that the Hadamard product $f_1\odot f_2$ of
series $f_1 \in \bK_{C_1}\langle\langle X_n\rangle\rangle$ and $f_2 \in
\bK_{C_2}\langle\langle X_n\rangle\rangle$ for different $C_1$ and~$C_2$
belongs to both rings, because the support of $f\odot g$ is the
intersection of the supports of $f$ and~$g$.

Observe that rational functions from $\bK(x_1,\dots,x_n)$, henceforth
denoted~$\bK(X_n)$, can often be expanded as elements of different rings
$\bK_{C_1}\langle\langle X_n\rangle\rangle$ and $\bK_{C_2}\langle\langle
X_n\rangle\rangle$, potentially with trivial intersection~$C_1 \cap C_2 =
\{0\}$.  This should not be understood as implying that the rational
function is part of the intersection of both rings.  In fact, we should
refrain from viewing rational functions as elements of series rings: they
only have images in those rings.  As a consequence, residue and
positive-part extractions of a rational function only make sense with
respect to a given cone expansion.  On the other hand, given two line-free
cones $C\subseteq C'$, we can safely identify $\bK_C[[X_n]]$ as a subring of
$\bK_{C'}[[X_n]]$, each given series having the same coefficient family with
regard to both expansions.  The situation is the same with
$\bK_C\langle\langle X_n\rangle\rangle$ that we identify as a subring of
$\bK_{C'}\langle\langle X_n\rangle\rangle$.  Similar identifications occur
when extending the variable set: for instance, we will freely identify
$\bK_C[[X_n]]$ and $\bK_{C\times\{0\}^m}[[X_{n+m}]]$.

To avoid blind identification of a rational function with a series, as the
latter must depend on the cone used, we introduce the following notation:
for a rational function~$P/Q \in \bK(X_n)$ which admits an expansion with
respect to a line-free cone~$C$, we write~$[P/Q]_C$ for this expansion in
$\bK_C\langle\langle X_n\rangle\rangle$.  As a consequence
of~\cite[Thm~12]{AparicioMonforteKauers-2013-FLS}, such an expansion exists
if and only if $Q$~can be factored in the form $x_1^{i_1}\cdots x_n^{i_n}
\tilde Q$ for $\tilde Q \in \bK_C[[X_n]]$ with non-zero constant term.

There will be situations where we need to know the existence of a cone~$C$
with respect to which a certain rational function can be expanded without
having to know the exact cone~$C$.  This will be the case
in Section~\ref{sec:ann-by-ct} below, when we link creative telescoping to series
expansions: we will want to know that certificates of creative telescoping
can be expanded over some cone~$C$ without having to observe them to be
able to make the cone~$C$ explicit.  To this end, we borrow
from~\cite{AparicioMonforteKauers-2013-FLS} the definition of more rings
and fields, defined with respect to a monomial order.

We start with a fixed monomial order~$\preccurlyeq$ on the monomials
in~$x_1,\dots,x_n$ (with exponents in~$\set Z$), that is, with a total
order that is monotonic with respect to product: for all $a$, $b$, and~$c$,
$a \preccurlyeq b$ implies $ac \preccurlyeq bc$.

\begin{definition}
A cone~$C$ is \emph{compatible\/} with the monomial order~$\preccurlyeq$ if
the apex of~$C$ coincides with its minimal element with respect
to~$\preccurlyeq$.
\end{definition}

Given~$\preccurlyeq$, the union over all cones~$C$ compatible
with~$\preccurlyeq$ of the rings $\bK_C[[X_n]]$ is a ring, denoted
$\bK_\preccurlyeq[[X_n]]$ in what follows.  The set $\bK_\preccurlyeq((X_n)) :=
\bigcup_{(i_1,\dots,i_n)\in\set Z^n} x_1^{i_1}\cdots x_n^{i_n}
\bK_\preccurlyeq[[X_n]]$ now forms another ring.  This ring is even a
field~\cite[Thm~15]{AparicioMonforteKauers-2013-FLS}.  Now, for a fixed
cone~$C$ compatible with the monomial order~$\preccurlyeq$, the ring
$\bK_C[[X_n]]$ is a subring of both $\bK_C\langle\langle X_n\rangle\rangle$ and
$\bK_\preccurlyeq[[X_n]]$, and both of these rings are subrings of the field
$\bK_\preccurlyeq((X_n))$. All those inclusions are canonical, in the sense that they preserve
the coefficients.
In particular, a rational function~$F$ of the field~$\bK(X_n)$ maps to a
well-defined series of~$\bK_\preccurlyeq((X_n))$, which we will
denote~$[F]_\preccurlyeq$.  When $C$~is compatible with~$\preccurlyeq$, we
will do the identification $[F]_C = [F]_\preccurlyeq$.  As a consequence,
the field $[\bK(X_n)]_\preccurlyeq=\{[F]_\preccurlyeq:F\in \bK(X_n)\}$ is a subfield of $\bK_\preccurlyeq((X_n))$.

\subsection{Positive parts as residues}
\label{sec:pp-as-res}

For a given series $f\in \bK_C\langle\langle X_n\rangle\rangle$, the positive
part $[x_1^>\cdots x_n^>] f$ can be obtained by taking the Hadamard product
with the expansion as a geometric series in $\bK_{\Rplus^n}[[X_n]]$ of $\frac
{x_1\cdots x_n} {(1-x_1)\cdots(1-x_n)}$.  We will argue that the usual
residue formula for computing Hadamard products holds in this setting.

\begin{lem}\label{thm:cone-product}
Let $\pi_1\colon\set R^n\to\set R^k$ and $\pi_2\colon\set R^n\to\set
R^{n-k}$ be the projections to the first $k$ and last $n-k$ coordinates,
respectively, so that we have $u=(\pi_1(u),\pi_2(u))$ for every $u\in\set
R^n$.  Each of the following constructions is a line-free cone:
\begin{enumerate}
\item the set $C_1 \star C_2 := \bigl\{(u, v-u) : u\in C_1,v\in
  C_2\bigr\}\subseteq\set R^{2n}$ for line-free cones $C_1,C_2\subseteq\set
  R^n$,
\item the set $C_1 \starkn C_2 := \bigl\{(u, \pi_2(v), \pi_1(v)-u) : u\in
  C_1,v\in C_2\bigr\}\subseteq\set R^{k+n}$ for line-free cones
  $C_1\subseteq\set R^k$ and $C_2\subseteq\set R^n$,
\item the set $C_1 \starnk C_2 := \bigl\{(u, v-\pi_1(u)) : u\in C_1,v\in
  C_2\bigr\}\subseteq\set R^{n+k}$ for line-free cones $C_1\subseteq\set
  R^n$ and $C_2\subseteq\set R^k$.
\end{enumerate}
\end{lem}

\begin{proof}
For the first construction, given $u_1,u_2\in C_1$, $v_1,v_2\in C_2$, and
$\lambda,\mu\geq0$, $u := \lambda u_1+\mu u_2$ and $v := \lambda v_1+\mu
v_2$ are in $C_1$ and~$C_2$, respectively, because $C_1,C_2$ are cones.
Consequently, $\lambda (u_1, v_1-u_1) + \mu (u_2, v_2-u_2) = (u, v-u)$ is
in~$C_1\star C_2$, and the latter is a cone.  Next, given $(u, v-u)$
in~$C_1 \star C_2$, if $- (u, v-u)$ is in~$C_1\star C_2$, too, then it can
be written $(u', v'-u')$.  From $(u, v-u) = - (u', v'-u')$ follows $u=-u'$,
which implies $u=u'=0$ because $C_1$ is line-free.  This, in turn, implies
that $v=-v'$, which implies that $v=v'=0$ because $C_2$ is line-free.
Therefore, $C_1\star C_2$ is line-free.

The same argument applies with minor changes to the two other
constructions.
\end{proof}

As a consequence, each of $\bK_{C_1\star C_2}\langle\langle
X_n,Y_n\rangle\rangle$, $\bK_{C_1 \starkn C_2}\langle\langle
X_n,Y_k\rangle\rangle$, and $\bK_{C_1 \starnk C_2}\langle\langle
X_n,Y_k\rangle\rangle$ is a bona fide ring for any two line-free cones
$C_1,C_2$ of $\set R^n$ or $\set R^k$ (as required).  Remark the intended
choice to use~$X_n$ in the case $C_1 \subseteq \set Z^k$ as well, and
not~$X_k$, in accordance to Equation~\eqref{eq:semi-hadam-k-n-as-res}
below.  As well, observe $C_1 \star C_2 = C_1 \underset{n\leq n}{\star} C_2
= C_1 \underset{n\geq n}{\star} C_2$, and similar identities on the level
of rings.

\begin{lem}\label{thm:hadam-as-res}
\emph{(i)} Let $C_1,C_2\subseteq\set R^n$ be two line-free cones.  Then,
for any $f\in \bK_{C_1}\langle\langle X_n\rangle\rangle$ and $g\in
\bK_{C_2}\langle\langle X_n\rangle\rangle$,
\begin{equation}\label{eq:hadam-as-res}
f \odot g =
\Res_{y_1,\dots,y_n}
  \frac1{y_1\cdots y_n}
  f\left(\frac{x_1}{y_1},\dots,\frac{x_n}{y_n}\right)
  g(y_1,\dots,y_n) ,
\end{equation}
where the argument of the residue is understood as a product in $\bK_{C_1\star
  C_2}\langle\langle X_n,Y_n\rangle\rangle$.

\emph{(ii)} For~$k\leq n$, let $C_1\subseteq\set R^k$ and $C_2\subseteq\set
R^n$ be two line-free cones.  Then, for any $f\in \bK_{C_1}\langle\langle
X_k\rangle\rangle$ and $g\in \bK_{C_2}\langle\langle X_n\rangle\rangle$,
\begin{equation}\label{eq:semi-hadam-k-n-as-res}
f \odotkn g =
\Res_{y_1,\dots,y_k}
  \frac1{y_1\cdots y_k}
  f\left(\frac{x_1}{y_1},\dots,\frac{x_k}{y_k}\right)
  g(y_1,\dots,y_k,x_{k+1},\dots,x_n) ,
\end{equation}
where the argument of the residue is understood as a product in $\bK_{C_1 \starkn
  C_2}\langle\langle X_n,Y_k\rangle\rangle$.

\emph{(iii)} For~$k\leq n$, let $C_1\subseteq\set R^n$ and
$C_2\subseteq\set R^k$ be two line-free cones.  Then, for any $f\in
\bK_{C_1}\langle\langle X_n\rangle\rangle$ and $g\in \bK_{C_2}\langle\langle
X_k\rangle\rangle$,
\begin{equation}\label{eq:semi-hadam-n-k-as-res}
f \odotnk g =
\Res_{y_1,\dots,y_k}
  \frac1{y_1\cdots y_k}
  f\left(\frac{x_1}{y_1},\dots,\frac{x_k}{y_k},x_{k+1},\dots,x_n\right)
  g(y_1,\dots,y_k) ,
\end{equation}
where the argument of the residue is understood as a product in $\bK_{C_1 \starnk
  C_2}\langle\langle X_n,Y_k\rangle\rangle$.
\end{lem}

Before the proof, observe that \eqref{eq:semi-hadam-k-n-as-res}
and~\eqref{eq:semi-hadam-n-k-as-res} cannot just be obtained as
specializations of~\eqref{eq:hadam-as-res}.

\begin{proof}
We start by proving~\eqref{eq:hadam-as-res}.  Set $C : = C_1\star
C_2\subseteq\set R^{2n}$, which is a line-free cone by
Lemma~\ref{thm:cone-product}.  Now, observe that $f(x_1/y_1,\dots,x_n/y_n)$
and $g(y_1,\dots,y_n)$ are in $\bK_C\langle\langle X_n,Y_n\rangle\rangle$,
respectively because $\bigl\{(u, -u) : u\in C_1\bigr\} \subseteq C$ and
because $\bigl\{(0, v) : v\in C_2\bigr\} \subseteq C$. So, the argument of
the residue in~\eqref{eq:hadam-as-res} is a well-defined product.  Call
it~$h$.  To see that $\Res_{y_1,\dots,y_n} h = f \odot g$, calculate:
\begin{multline}\label{eq:hadam-calc}
h =
  \frac1{y_1\cdots y_n}
  \biggl(\sum_{u_1,\dots,u_n} f_{u_1,\dots,u_n} \frac{x_1^{u_1}\cdots x_n^{u_n}}{y_1^{u_1} \cdots y_n^{u_n}}\biggr)
  \biggl(\sum_{v_1,\dots,v_n} g_{v_1,\dots,v_n} y_1^{v_1}\cdots y_n^{v_n}\biggr) = \\
\sum_{i_1,\dots,i_n} \biggl(
  \sum_{(v_1,\dots,v_n)-(u_1,\dots,u_n)=(i_1,\dots,i_n)}
    f_{u_1,\dots,u_n}g_{v_1,\dots,v_n} x_1^{u_1}\cdots x_n^{u_n}
  \biggr) y_1^{i_1-1}\cdots y_n^{i_n-1} ,
\end{multline}
where the sums are meant to extend over all integers with the understanding
that $f_{u_1,\dots,u_n}$ and $g_{v_1,\dots,v_n}$ are zero outside of the
respective finite unions of translated cones.
Applying~$\Res_{y_1,\dots,y_n}$ to this identity selects the term for which
$i_1 = \dots = i_n = 0$, forcing $(v_1,\dots,v_n) = (u_1,\dots,u_n)$.
Hence we have:
\begin{equation*}
\Res_{y_1,\dots,y_n} h =
  \sum_{(u_1,\dots,u_n)} f_{u_1,\dots,u_n}g_{u_1,\dots,u_n} x_1^{u_1}\cdots x_n^{u_n} =
  f \odot g ,
\end{equation*}
as was to be proved to justify~\eqref{eq:hadam-as-res}.

To prove~\eqref{eq:semi-hadam-k-n-as-res}, set $C' : = C_1 \starkn
C_2\subseteq\set R^{k+n}$, which is a line-free cone by
Lemma~\ref{thm:cone-product}.  Now, observe that $f(x_1/y_1,\dots,x_k/y_k)$
and $g(y_1,\dots,y_k,x_{k+1},\dots,x_n)$ are in $\bK_{C'}\langle\langle
X_n,Y_k\rangle\rangle$, respectively because $\bigl\{(u, 0, -u) : u\in
C_1\bigr\} \subseteq C'$ and because $\bigl\{(0, \pi_2(v), \pi_1(v)) : v\in
C_2\bigr\} \subseteq C'$. So, the argument of the residue
in~\eqref{eq:semi-hadam-k-n-as-res} is a well-defined product.  The proof
of~\eqref{eq:semi-hadam-k-n-as-res} then follows calculations analogous
to~\eqref{eq:hadam-calc}.

To prove~\eqref{eq:semi-hadam-n-k-as-res}, set $C'' := C_1 \starnk
C_2\subseteq\set R^{n+k}$, which is a line-free cone by
Lemma~\ref{thm:cone-product}.  Now, observe that
$f(x_1/y_1,\dots,x_k/y_k,x_{k+1},\dots,x_n)$ and $g(y_1,\dots,y_k)$ are in
$\bK_{C''}\langle\langle X_n,Y_k\rangle\rangle$, respectively because
$\bigl\{(u, -\pi_1(u)) : u\in C_1\bigr\} \subseteq C''$ and because
$\bigl\{(0, v) : v\in C_2\bigr\} \subseteq C''$. So, the argument of the
residue in~\eqref{eq:semi-hadam-n-k-as-res} is a well-defined product.  The
proof of~\eqref{eq:semi-hadam-n-k-as-res} then follows calculations
analogous to~\eqref{eq:hadam-calc}.
\end{proof}

Recall the notation~$[P/Q]_C$ to denote the expansion of a rational
function with respect to a cone.

\begin{lem}\label{thm:ppart-as-res}
For every line-free cone $C\subseteq\set R^n$, every $\phi\in
\bK_C\langle\langle X_n\rangle\rangle$, and every
$k\in\{1,\dots,n\}$, we have:
\begin{align}
[x_1^>\cdots x_k^>]\phi &= \label{eq:ppart-as-res-plain-phi}
  \Res_{y_1,\dots,y_k}
    \frac{1}{y_1\cdots y_k}
    \left[\frac{\frac{x_1}{y_1}\cdots\frac{x_k}{y_k}}{(1-\frac{x_1}{y_1})\cdots(1-\frac{x_k}{y_k})}\right]_{\Rplus^k\times\Rminus^k}
    \phi(y_1,\dots,y_k,x_{k+1},\dots,x_n) \\
{} &= \label{eq:ppart-as-res-substd-phi}
  \Res_{y_1,\dots,y_k}
    \frac1{y_1\cdots y_k}
    \phi\left(\frac{x_1}{y_1},\dots,\frac{x_k}{y_k},x_{k+1},\dots,x_n\right)
    \left[\frac{y_1\cdots y_k}{(1-y_1)\cdots(1-y_k)}\right]_{\Rplus^k} ,
\end{align}
where the brackets around rational functions in
\eqref{eq:ppart-as-res-plain-phi} and~\eqref{eq:ppart-as-res-substd-phi}
denote taking their expansions in, respectively,
$\bK_{\Rplus^k\times\Rminus^k}\langle\langle X_k,Y_k\rangle\rangle$ and
$\bK_{\Rplus^k}\langle\langle Y_k\rangle\rangle$.
\end{lem}

\begin{proof}
Fix $C'$~to $\Rplus^k$ and $\psi\in \bK_{C'}\langle\langle X_k\rangle\rangle$
to the geometric-series expansion of $\frac {x_1\cdots x_k}
{(1-x_1)\cdots(1-x_k)}$.  The desired positive part can then be represented
in two ways as variants of Hadamard products: $[x_1^>\cdots x_k^>]\phi =
\psi \odotkn \phi = \phi \odotnk \psi$.  We then use
Lemma~\ref{thm:hadam-as-res} twice:

\begin{itemize}

\item Firstly, by~\eqref{eq:semi-hadam-k-n-as-res}, $[x_1^>\cdots
  x_k^>]\phi = \Res_{y_1,\dots,y_k} \frac{1}{y_1\cdots y_k}
  \psi\bigl(\frac{x_1}{y_1},\dots,\frac{x_k}{y_k}\bigr)
  \phi(y_1,\dots,y_k,x_{k+1},\dots,x_n)$, where the product is in $\bK_{C'
    \starkn C}\langle\langle X_n,Y_k\rangle\rangle$.  More specifically,
  $\psi\bigl(\frac{x_1}{y_1},\dots,\frac{x_k}{y_k}\bigr)$ is in $\bK_{C'
    \starkn \{0\}^n}\langle\langle X_n,Y_k\rangle\rangle$, and with the
  proper identification in $\bK_{\Rplus^k\star\{0\}^k}\langle\langle
  X_k,Y_k\rangle\rangle$, hence \eqref{eq:ppart-as-res-plain-phi} and the
  announced expansion.

\item Secondly, by~\eqref{eq:semi-hadam-n-k-as-res}, $[x_1^>\cdots
  x_k^>]\phi = \Res_{y_1,\dots,y_k} \frac{1}{y_1\cdots y_k}
  \phi\bigl(\frac{x_1}{y_1},\dots,\frac{x_k}{y_k},x_{k+1},\dots,x_n\bigr)
  \psi(y_1,\dots,y_k)$, where the product is in $\bK_{C \starnk
    C'}\langle\langle X_n,Y_n\rangle\rangle$.  More specifically,
  $\psi(y_1,\dots,y_k)$ is in $\bK_{\{0\}^n \starnk C'}\langle\langle
  X_n,Y_k\rangle\rangle$, and with the proper identification in
  $\bK_{C'}\langle\langle Y_k\rangle\rangle$, hence
  \eqref{eq:ppart-as-res-substd-phi} and the announced expansion.

\end{itemize}

\end{proof}

\subsection{Specializations of positive parts}
\label{sec:spec-pp}

Next, we would like to justify a formula to represent the series % MK: avoid overfull hbox
$\bigl([x_1^>\cdots x_k^>]f\bigr)\bigr|_{x_1=\alpha_1,\dots,x_k=\alpha_k}$ as a Hadamard product, where
$\alpha_1,\dots,\alpha_k$ are fixed elements of~$\bK$.  For the evaluation of the
positive part at some arbitrary field elements to make sense, we need to
impose a restriction on the cones, which we will express using the
following notion.

\begin{definition}
  Let $\pi_2\colon\set R^n\to\set R^{n-k}$ be the projection to the last
  $n-k$ coordinates.
We say that two line-free cones $C_1,C_2\subseteq\set R^n$ are \emph{in
  opposition with respect to the first $k$ variables\/} if
$\pi_2(C_1),\pi_2(C_2)$ are line-free cones and $C_1\cap C_2\cap(\set
R^k\times\{0\}^{n-k})=\{0\}^n$.
\end{definition}

\begin{lem}\label{thm:bounded-affine-intersection}
Let $C_1$ and~$C_2$ in~$\set R^n$ be in opposition with respect to the
first $k\leq n$ variables, and let $u,v\in\set R^n$.  Then, the set
$(C_1+u)\cap(C_2+v)\cap(\set R^k\times\{0\}^{n-k})$ is bounded.
\end{lem}

\begin{proof}
Given $w,d \in \set R^n$, denote by $R_{w,d}$ the set $\{w + cd :
c\geq0\}$.  We recall without proof that: \ \emph{(i)\/}~for any~$R_{w,d}$
contained in a cone~$C$, both $w$ and~$d$ are in~$C$; \ \emph{(ii)\/}~any
unbounded convex set~$\bK$ contains some~$R_{w,d}$ with nonzero~$d$;
\ \emph{(iii)\/}~as a consequence, for any cone~$C$, any time $R_{w,d}
\subseteq w' + C$, we have $w-w',d \in C$.

Now, if the (convex) set $(C_1+u)\cap(C_2+v)\cap(\set
R^k\times\{0\}^{n-k})$ was unbounded, it would contain some~$R_{w,d}$ for a
nonzero~$d$.  We would have in particular that $d$~would be in $C_1$, $C_2$,
and~$\set R^k\times\{0\}^{n-k}$, implying $d=0$.  This would contradict
that $C_1$ and~$C_2$ are in opposition with respect to the first~$k$
variables.
\end{proof}

Lemma~\ref{lem:wf-spec-hadam} below gives a sufficient condition that
allows a Hadamard product to specialize at $x_1=\alpha_1$, \dots, $x_k=\alpha_k$ in a
well defined way.  Two potential obstructions to this are:
\ \emph{(i)\/}~situations in which $f\odot g$~involves either of
$x_1,\dots,x_k$ with a negative exponent; \ \emph{(ii)\/}~situations in
which the supports of $f$ and~$g$ make it possible that the Hadamard
product should involve some monomial in $x_{k+1},\dots,x_n$ having as
coefficient an infinite sum in $x_1,\dots,x_k$ that does not converge as a
series at $x_1=\alpha_1,\dots,x_k=\alpha_k$.  In the lemma, the cones being in
opposition ensures the finiteness of all sums under consideration
in~\emph{(ii)}, while case~\emph{(i)\/} is excluded by imposing~$\alpha_j\neq0$
when necessary.  To state conditions for the latter point, for~$1\leq j\leq
k\leq n$, for a series~$s$ in~$x_1,\dots,x_k$ and a series~$t$
in~$x_1,\dots,x_n$, we introduce two predicates $H_j^n(t)$ and~$H_j^{k\leq
  n}(s,t)$ that express, respectively, that $t$~cannot be evaluated at~$x_j
= 0$ and that the (generally asymmetric) Hadamard product of $s$ and~$t$
cannot be evaluated at~$x_j = 0$; they read:
\begin{quote}\begin{description}
\item[$H_j^n(t)$] : There exists a monomial in~$x_1,\dots,x_n$ with
  negative exponent of~$x_j$ that occurs with non-zero coefficient in the
  series~$t$.
\item[$H_j^{k\leq n}(s,t)$] : There exists a monomial~$m$
  in~$x_1,\dots,x_k$ with negative exponent of~$x_j$ and a monomial~$m'$
  in~$x_{k+1},\dots,x_n$, such that $m$~occurs with non-zero coefficient in
  the series~$s$ and $mm'$~occurs with non-zero coefficient in the
  series~$t$.
\end{description}\end{quote}
Note that $H_j^{k\leq n}(s,t)$ implies $H_j^k(s)$ and~$H_j^n(t)$.

\begin{lem}\label{lem:wf-spec-hadam}
\emph{(i)}~Let $C_1,C_2\subseteq\set R^n$ be in opposition with respect to
the first $k$ variables.  Further, let $f\in \bK_{C_1}\langle\langle
X_n\rangle\rangle$ and $g\in \bK_{C_2}\langle\langle X_n\rangle\rangle$.
Then, $(f\odot g)|_{x_1=\alpha_1,\dots,x_k=\alpha_k}$ is well-defined for every
$\alpha_1,\dots,\alpha_k\in \bK$, provided that for each~$1\leq j\leq k$, $\alpha_j\neq0$
if~$H_j^{n\leq n}(f,g)$.

\emph{(ii)}~Let $C_1\subseteq\set R^k$ and $C_2\subseteq\set R^n$ be such
that $C_1\times\{0\}^{n-k}$ and~$C_2$ are in opposition with respect to the
first $k$ variables.  Further, let $f\in \bK_{C_1}\langle\langle
X_k\rangle\rangle$ and $g\in \bK_{C_2}\langle\langle X_n\rangle\rangle$.
Then, $(f\odotkn g)|_{x_1=\alpha_1,\dots,x_k=\alpha_k}$ is well-defined for every
$\alpha_1,\dots,\alpha_k\in \bK$, provided that for each~$1\leq j\leq k$, $\alpha_j\neq0$
if~$H_j^{k\leq n}(f,g)$.

\emph{(iii)}~Let $C_1\subseteq\set R^n$ and $C_2\subseteq\set R^k$ be such
that $C_1$ and~$C_2\times\{0\}^{n-k}$ are in opposition with respect to the
first $k$ variables.  Further, let $f\in \bK_{C_1}\langle\langle
X_n\rangle\rangle$ and $g\in \bK_{C_2}\langle\langle X_k\rangle\rangle$.
Then, $(f\odotnk g)|_{x_1=\alpha_1,\dots,x_k=\alpha_k}$ is well-defined for every
$\alpha_1,\dots,\alpha_k\in \bK$, provided that for each~$1\leq j\leq k$, $\alpha_j\neq0$
if~$H_j^{k\leq n}(g,f)$.
\end{lem}

\begin{proof}
\emph{(i)\/}~Suppose first that there are vectors $e_1,e_2\in\set R^n$ such
that the support of $f$ is contained in $e_1+C_1$ and the support of $g$ is
contained in $e_2+C_2$.  Fix $i=(i_{k+1},\dots,i_n)\in\set Z^{n-k}$ and
introduce~$i' = (0,\dots,0,i_{k+1},\dots,i_n)\in\set Z^n$.  From $C_1$ and
$C_2$ being in opposition with respect to the first $k$ variables, it
follows by Lemma~\ref{thm:bounded-affine-intersection} that
$(e_1-i'+C_1)\cap(e_2-i'+C_2)\cap(\set R^k\times\{0\}^{n-k})$ is bounded.
In other words, $(e_1-i'+C_1)\cap(e_2-i'+C_2)\cap\pi_2^{-1}(0)$ is bounded,
and, after shifting by~$i'$, so is the set
$(e_1+C_1)\cap(e_2+C_2)\cap\pi_2^{-1}(i)$.  Hence, there are at most
finitely many vectors $(i_1,\dots,i_k)\in\set Z^k$ such that the
coefficient of $x_1^{i_1}\dots x_n^{i_n}$ in $f\odot g$ is nonzero.

In the general case, there exist finite families
$e_1^{(1)},\dots,e_1^{(r)}$ and $e_2^{(1)},\dots,e_2^{(s)}$ such that the
support of~$f$ is contained in $\bigcup_{i=1}^r e_1^{(i)} + C_1$ and such
that the support of~$g$ is contained in $\bigcup_{i=1}^r e_2^{(i)} + C_2$.
Given a fixed $i=(i_{k+1},\dots,i_n)\in\set Z^{n-k}$ again, the finiteness
of the number of vectors $(i_1,\dots,i_k)\in\set Z^k$ such that the
coefficient of $x_1^{i_1}\dots x_n^{i_n}$ in $f\odot g$ is nonzero still
holds, by distributivity and the previous argument.

The constraints on the~$\alpha_i$ for a well-defined substitution follow in all
cases.

\emph{(ii)\/}~The proof is like the case~\emph{(i)}, with minimal changes:
considering $e_1\in\set R^k$ instead of $e_1\in\set R^n$, replacing~$C_1$
by~$C_1\times\{0\}^{n-k}$ in the use of
Lemma~\ref{thm:bounded-affine-intersection}, and replacing~$f\odot g$
by~$f\odotkn g$ in the conclusion.

\emph{(iii)\/}~The proof is like the case~\emph{(i)}, with minimal changes:
considering $e_2\in\set R^k$ instead of $e_2\in\set R^n$, replacing~$C_2$
by~$C_2\times\{0\}^{n-k}$ in the use of
Lemma~\ref{thm:bounded-affine-intersection}, and replacing~$f\odot g$
by~$f\odotnk g$ in the conclusion.
\end{proof}

The following lemma generalizes Lemma~\ref{thm:cone-product} to cones in
opposition.

\begin{lem}\label{thm:opposing-cone-product}
For any~$m\geq k\geq 0$, let $\tau\colon\set R^m\to\set R^{m-k}$ be the
projection to the last $m-k$ coordinates.

\emph{(i)}~Given two cones $C_1,C_2\subseteq\set R^n$ that are in
opposition with respect to the first $k$ variables, the set $\tau(C_1\star
C_2) = \bigl\{(\pi_2(u), v-u) : u\in C_1,v\in C_2\bigr\}\subseteq\set
R^{(n-k)+n}$ is a line-free cone.

\emph{(ii)}~Given two cones $C_1\subseteq\set R^k$ and $C_2\subseteq\set
R^n$ such that $C_1\times\{0\}^{n-k}$ and~$C_2$ are in opposition with
respect to the first $k$ variables, the set $\tau(C_1 \starkn C_2) =
\bigl\{(\pi_2(v), \pi_1(v)-u) : u\in C_1,v\in C_2\bigr\}\subseteq\set
R^{k+n}$ is a line-free cone.

\emph{(iii)}~Given two cones $C_1\subseteq\set R^n$ and $C_2\subseteq\set
R^k$ such that $C_1$ and~$C_2\times\{0\}^{n-k}$ are in opposition with
respect to the first $k$ variables, the set $\tau(C_1 \starnk C_2) =
\bigl\{(\pi_2(u), v-\pi_1(u)) : u\in C_1,v\in C_2\bigr\}\subseteq\set
R^{n+k}$ is a line-free cone.
\end{lem}

\begin{proof}
\emph{(i)\/}~The proof that $\tau(C_1\star C_2)$ is a cone is similar to
the proof for the case of~$C_1 \star C_2$ in Lemma~\ref{thm:cone-product}.
We show that it is line-free: assume there is $w \in \tau(C_1\star C_2)$
such that $-w \in \tau(C_1\star C_2)$; we proceed to prove that $w=0$.  By
assumption, there exist $u,u' \in C_1$ and $v,v' \in C_2$ such that $w
= (\pi_2(u), v-u)$ and $-w = (\pi_2(u'), v'-u')$, and therefore such that
$(\pi_2(u), v-u) = - (\pi_2(u'), v'-u')$.  Then $\pi_2(u)=-\pi_2(u')$
implies $\pi_2(u)=\pi_2(u')=0$, because $\pi_2(C_1)$ is line-free.  Next,
$v-u=-(v'-u')$ implies $\pi_2(v-u)=\pi_2(-(v'-u'))$, and therefore
$\pi_2(v)=-\pi_2(v')$, because $\pi_2(u)=\pi_2(u')=0$.  It follows that
$\pi_2(v)=\pi_2(v')=0$ because $\pi_2(C_2)$ is line-free.  Next, observe
that $v+v'\in C_2\cap(\set R^k\times\{0\}^{n-k})$, because
$\pi_2(v)=\pi_2(v')=0$, and that $u+u'\in C_1\cap(\set
R^k\times\{0\}^{n-k})$, because $\pi_2(u)=\pi_2(u')=0$.  From
$v-u=-(v'-u')$ follows $v+v'=u+u' \in C_1\cap C_2\cap(\set
R^k\times\{0\}^{n-k})$.  As $C_1$ and~$C_2$ are in opposition with respect
to the first $k$ variables, both $v+v'$ and $u+u'$ are zero.  Since $C_1$,
resp.~$C_2$, is line-free, this finally implies that $u=u'=0$, resp.~that
$v=v'=0$, thus that~$w=0$.

\emph{(ii)\/}~The proof that $\tau(C_1 \starkn C_2)$ is a cone is similar
to the proof for the case of~$C_1 \starkn C_2$ in
Lemma~\ref{thm:cone-product}.  We show that it is line-free: assume there
is $w \in \tau(C_1 \starkn C_2)$ such that $-w \in \tau(C_1 \starkn C_2)$;
we proceed to prove that $w=0$.  By assumption, there exist $u,u' \in
C_1$ and $v,v' \in C_2$ such that $w = (\pi_2(v), \pi_1(v)-u)$ and $-w =
(\pi_2(v'), \pi_1(v')-u')$, and therefore such that $(\pi_2(v), \pi_1(v)-u)
= - (\pi_2(v'), \pi_1(v')-u')$.  Then $\pi_2(v)=-\pi_2(v')$ implies
$\pi_2(v)=\pi_2(v')=0$, because $\pi_2(C_2)$ is line-free.  Next,
$\pi_1(v)-u=-(\pi_1(v')-u')$ implies $\pi_1(v+v')=u+u'$.  Therefore, $v+v'
= \pi_1(v+v') + \pi_2(v+v') = (u+u') + 0$.  As $C_1\times\{0\}^{n-k}$
and~$C_2$ are in opposition with respect to the first $k$ variables, both
$v+v'$ and $u+u'$ are zero.  Since $C_1$, resp.~$C_2$, is line-free, this
finally implies that $u=u'=0$, resp.~that $v=v'=0$, thus that~$w=0$.

\emph{(iii)\/}~The proof that $\tau(C_1 \starnk C_2)$ is a cone is similar
to the proof for the case of~$C_1 \starnk C_2$ in
Lemma~\ref{thm:cone-product}.  We show that it is line-free: assume there
is $w \in \tau(C_1 \starnk C_2)$ such that $-w \in \tau(C_1 \starnk C_2)$;
we proceed to prove that $w=0$.  By assumption, there exist $u,u' \in
C_1$ and $v,v' \in C_2$ such that $w = (\pi_2(u), v-\pi_1(u))$ and $-w =
(\pi_2(u'), v'-\pi_1(u'))$, and therefore such that $(\pi_2(u), v-\pi_1(u))
= - (\pi_2(u'), v'-\pi_1(u'))$.  Then $\pi_2(u)=-\pi_2(u')$ implies
$\pi_2(u)=\pi_2(u')=0$, because $\pi_2(C_1)$ is line-free.  Next,
$v-\pi_1(u)=-(v'-\pi_1(u'))$ implies $\pi_1(u+u')=v+v'$.  Therefore, $u+u'
= \pi_1(u+u') + \pi_2(u+u') = (v+v') + 0$.  As $C_1$
and~$C_2\times\{0\}^{n-k}$ are in opposition with respect to the first $k$
variables, both $u+u'$ and $v+v'$ are zero.  Since $C_1$, resp.~$C_2$, is
line-free, this finally implies that $u=u'=0$, resp.~that $v=v'=0$, thus
that~$w=0$.
\end{proof}

Hence, for any two line-free cones $C_1,C_2$ of $\set R^n$ or $\set R^k$
and with the relevant cones in opposition with respect to the first $k$
variables (as required), each of $\bK_{\tau(C_1\star C_2)}\langle\langle
X_{n\setminus k},Y_n\rangle\rangle$, $\bK_{\tau(C_1\starkn
  C_2)}\langle\langle X_{n\setminus k},Y_k\rangle\rangle$, and
$\bK_{\tau(C_1\starnk C_2)}\langle\langle X_{n\setminus k},Y_k\rangle\rangle$
is a well-defined ring, and a subring of, respectively, $\bK_{C_1\star
  C_2}\langle\langle X_n,Y_n\rangle\rangle$, $\bK_{C_1\starkn
  C_2}\langle\langle X_n,Y_k\rangle\rangle$, or $\bK_{C_1\starnk
  C_2}\langle\langle X_n,Y_k\rangle\rangle$.

The following lemma can now be viewed as a generalization of
Lemma~\ref{thm:hadam-as-res} above.

\begin{lem}\label{thm:spec-hadam-as-res}
\emph{(i)} Let $C_1,C_2\subseteq\set R^n$ be two cones that are in
opposition with respect to the first $k$ variables.  Let $f\in
\bK_{C_1}\langle\langle X_n\rangle\rangle$, $g\in \bK_{C_2}\langle\langle
X_n\rangle\rangle$, and $\alpha_1,\dots,\alpha_k,\beta_1,\dots,\beta_k\in \bK$ satisfy: for
each $1\leq j\leq k$, $\alpha_j\neq0$ if~$H_j^n(f)$ and $\beta_j\neq0$
if~$H_j^n(g)$.  Then,
\begin{multline}\label{eq:substd-hadam-as-res}
(f\odot g)\Bigr|_{x_1=\alpha_1\beta_1,\dots,x_k=\alpha_k\beta_k} = \\
\Res_{y_1,\dots,y_n}
  \frac1{y_1\cdots y_n}
  f\left(\frac{\alpha_1}{y_1},\dots,\frac{\alpha_k}{y_k},\frac{x_{k+1}}{y_{k+1}},\dots,\frac{x_n}{y_n}\right)
  g(\beta_1y_1,\dots,\beta_ky_k,y_{k+1},\dots,y_n) ,
\end{multline}
where the argument of the residue is understood as a product in $\bK_{\tau(C_1\star
  C_2)}\langle\langle X_{n\setminus k},Y_n\rangle\rangle$.

\emph{(ii)} For~$k\leq n$, let $C_1\subseteq\set R^k$ and $C_2\subseteq\set
R^n$ be two line-free cones, such that $C_1\times\{0\}^{n-k}$ and~$C_2$ are
in opposition with respect to the first $k$ variables.  Let $f\in
\bK_{C_1}\langle\langle X_k\rangle\rangle$, $g\in \bK_{C_2}\langle\langle
X_n\rangle\rangle$, and $\alpha_1,\dots,\alpha_k,\beta_1,\dots,\beta_k\in \bK$ satisfy: for
each $1\leq j\leq k$, $\alpha_j\neq0$ if~$H_j^k(f)$ and $\beta_j\neq0$
if~$H_j^n(g)$.
% Just a consequence: $a_jb_j\neq0$ if~$H_j^{k\leq n}(f,g)$
Then,
\begin{multline}\label{eq:substd-semi-hadam-k-n-as-res}
(f \odotkn g)\Bigr|_{x_1=\alpha_1\beta_1,\dots,x_k=\alpha_k\beta_k} = \\
\Res_{y_1,\dots,y_k}
  \frac1{y_1\cdots y_k}
  f\left(\frac{\alpha_1}{y_1},\dots,\frac{\alpha_k}{y_k}\right)
  g(\beta_1y_1,\dots,\beta_ky_k,x_{k+1},\dots,x_n) ,
\end{multline}
where the argument of the residue is understood as a product in $\bK_{\tau(C_1 \starkn
  C_2)}\langle\langle X_{n\setminus k},Y_k\rangle\rangle$.

\emph{(iii)} For~$k\leq n$, let $C_1\subseteq\set R^n$ and
$C_2\subseteq\set R^k$ be two line-free cones, such that $C_1$
and~$C_2\times\{0\}^{n-k}$ are in opposition with respect to the first $k$
variables.  Let $f\in \bK_{C_1}\langle\langle X_n\rangle\rangle$, $g\in
\bK_{C_2}\langle\langle X_k\rangle\rangle$, and
$\alpha_1,\dots,\alpha_k,\beta_1,\dots,\beta_k\in \bK$ satisfy: for each $1\leq j\leq k$,
$\alpha_j\neq0$ if~$H_j^n(f)$ and $\beta_j\neq0$ if~$H_j^k(g)$.
% Just a consequence: $a_jb_j\neq0$ if~$H_j^{k\leq n}(g,f)$
Then,
\begin{multline}\label{eq:substd-semi-hadam-n-k-as-res}
(f \odotnk g)\Bigr|_{x_1=\alpha_1\beta_1,\dots,x_k=\alpha_k\beta_k} = \\
\Res_{y_1,\dots,y_k}
  \frac1{y_1\cdots y_k}
  f\left(\frac{\alpha_1}{y_1},\dots,\frac{\alpha_k}{y_k},x_{k+1},\dots,x_n\right)
  g(\beta_1y_1,\dots,\beta_ky_k) ,
\end{multline}
where the argument of the residue is understood as a product in $\bK_{\tau(C_1 \starnk
  C_2)}\langle\langle X_{n\setminus k},Y_k\rangle\rangle$.
\end{lem}

Before the proof, observe that \eqref{eq:substd-semi-hadam-k-n-as-res}
and~\eqref{eq:substd-semi-hadam-n-k-as-res} cannot just be obtained as
specializations of~\eqref{eq:substd-hadam-as-res}.

\begin{proof}
To prove~\eqref{eq:substd-hadam-as-res}, assume that
$\alpha_1,\dots,\alpha_k,\beta_1,\dots,\beta_k\in \bK$ satisfy: for each $1\leq j\leq k$,
$\alpha_j\neq0$ if~$H_j^n(f)$ and $\beta_j\neq0$ if~$H_j^k(g)$.  As a first
consequence, for each $1\leq j\leq k$, $\alpha_j\beta_j\neq0$ if~$H_j^{n\leq
  n}(f,g)$, so that by Lemma~\ref{lem:wf-spec-hadam}, the left-hand side
of~\eqref{eq:substd-hadam-as-res} is well-defined.  As a second
consequence,
$f\bigl(\frac{\alpha_1}{y_1},\dots,\frac{\alpha_k}{y_k},\frac{x_{k+1}}{y_{k+1}},\dots,\frac{x_n}{y_n}\bigr)$
and $g(\beta_1y_1,\dots,\beta_ky_k,y_{k+1},\dots,y_n)$ are well defined, too, and
are thus in $\bK_{\tau(C_1\star C_2)}\langle\langle X_{n\setminus
  k},Y_n\rangle\rangle$.  In a way similar to~\eqref{eq:hadam-calc} in the
proof of Lemma~\ref{thm:hadam-as-res}, we reformulate their product~$h$ in
the form:
\begin{multline*}
h =
  \frac1{y_1\cdots y_n}
  \biggl(\sum_{u_1,\dots,u_n} f_{u_1,\dots,u_n} \frac{\alpha_1^{u_1}\cdots \alpha_k^{u_k} x_{k+1}^{u_{k+1}}\cdots x_n^{u_n}}{y_1^{u_1}\cdots y_n^{u_n}}\biggr)
  \biggl(\sum_{v_1,\dots,v_n} g_{v_1,\dots,v_n} \beta_1^{v_1} \cdots \beta_k^{v_k} y_1^{v_1}\cdots y_n^{v_n}\biggr) = \\
\sum_{i_1,\dots,i_n} \biggl(
  \sum_{(v_1,\dots,v_n)-(u_1,\dots,u_n)=(i_1,\dots,i_n)}
    f_{u_1,\dots,u_n}g_{v_1,\dots,v_n} (\alpha_1^{u_1}\beta_1^{v_1})\cdots (\alpha_k^{u_k}\beta_k^{v_k}) x_{k+1}^{u_{k+1}}\cdots x_n^{u_n}
  \biggr) y_1^{i_1-1}\cdots y_n^{i_n-1} .
\end{multline*}
Like in Lemma~\ref{thm:hadam-as-res}, extracting residues with respect to
$y_1,\dots,y_n$ selects the term for which $i_1 = \dots = i_n = 0$, forcing
$(v_1,\dots,v_n) = (u_1,\dots,u_n)$, which
yields~\eqref{eq:substd-hadam-as-res}.

The proofs or \eqref{eq:substd-semi-hadam-k-n-as-res} and
\eqref{eq:substd-semi-hadam-n-k-as-res} are similar, basing on analogous
calculations.
\end{proof}

Likewise, the following theorem can now be viewed as a generalization of
Lemma~\ref{thm:ppart-as-res} above.  For the proof, we proceed like for the
proof of that lemma, invoking Lemma~\ref{thm:spec-hadam-as-res} in place of
Lemma~\ref{thm:hadam-as-res}.

\begin{thm}\label{thm:spec-ppart-as-res}
For every line-free cone $C\subseteq\set R^n$ that is in opposition
to~$\Rplus^n$ with respect to the first $k$ coordinates, for every $\phi\in
\bK_C\langle\langle X_n\rangle\rangle$, and for every
$\lambda_1,\dots,\lambda_k\in \bK$, we have:
\begin{align}
\bigl([x_1^>\cdots x_k^>]&\phi\bigr)\Bigr|_{x_1=\lambda_1,\dots,x_k=\lambda_k} \nonumber \\
{} &= \label{eq:spec-ppart-as-res-plain-phi}
  \Res_{y_1,\dots,y_k}
    \frac1{y_1\cdots y_k}
    \left[\frac{\frac{\lambda_1}{y_1}\cdots\frac{\lambda_k}{y_k}}{(1-\frac{\lambda_1}{y_1})\cdots(1-\frac{\lambda_k}{y_k})}\right]_{\Rminus^k}
    \phi(y_1,\dots,y_k,x_{k+1},\dots,x_n) \\
{} &= \label{eq:spec-ppart-as-res-substd-phi}
  \Res_{y_1,\dots,y_k}
    \frac1{y_1\cdots y_k}
    \phi\left(\frac1{y_1},\dots,\frac1{y_k},x_{k+1},\dots,x_n\right)
    \left[\frac{\lambda_1y_1\cdots \lambda_ky_k}{(1-\lambda_1y_1)\cdots(1-\lambda_ky_k)}\right]_{\Rplus^k} ,
\end{align}
where the brackets around rational functions in
\eqref{eq:spec-ppart-as-res-plain-phi}
and~\eqref{eq:spec-ppart-as-res-substd-phi} denote taking their expansions
in, respectively, $\bK_{\Rminus^k}\langle\langle Y_k\rangle\rangle$ and
$\bK_{\Rplus^k}\langle\langle Y_k\rangle\rangle$.
\end{thm}

\begin{proof}
Fix $C'$~to $\Rplus^k$ and $\psi\in \bK_{C'}\langle\langle X_k\rangle\rangle$
to the geometric series expansion of $\frac {x_1\cdots x_k}
{(1-x_1)\cdots(1-x_k)}$.  The desired positive part can then be represented
in two ways as variants of Hadamard products:
\begin{equation*}
\bigl([x_1^>\cdots x_k^>]\phi\bigr)\bigr|_{x_1=\lambda_1,\dots,x_k=\lambda_k} =
  \bigl(\psi \odotkn \phi\bigr)\bigr|_{x_1=\lambda_1,\dots,x_k=\lambda_k} =
  \bigl(\phi \odotnk \psi\bigr)\bigr|_{x_1=\lambda_1,\dots,x_k=\lambda_k} .
\end{equation*}
We then use Lemma~\ref{thm:spec-hadam-as-res} twice, which applies as
$C$~is in opposition to~$\Rplus^n$ with respect to the first $k$
coordinates, thus to $C'\times\{0\}^{n-k}$ as well:

\begin{itemize}

\item Firstly, Lemma~\ref{thm:spec-hadam-as-res}\,(ii) can be used with
  $C_1=C'$, $C_2=C$, $f=\psi$, $g=\phi$, $\alpha_i=\lambda_i$, $\beta_i=1$, without
  restriction on~$\lambda_i\in \bK$ as $\psi$~involves no monomial with
  negative exponent.  Equation~\eqref{eq:substd-semi-hadam-k-n-as-res} then
  reads
  \begin{equation*}
  \bigl([x_1^>\cdots x_k^>]\phi\bigr)\Bigr|_{x_1=\lambda_1,\dots,x_k=\lambda_k} =
    \Res_{y_1,\dots,y_k} \frac{1}{y_1\cdots y_k}
    \psi\left(\frac{\lambda_1}{y_1},\dots,\frac{\lambda_k}{y_k}\right)
    \phi(y_1,\dots,y_k,x_{k+1},\dots,x_n) ,
  \end{equation*}
  where $\psi\bigl(\frac{\lambda_1}{y_1},\dots,\frac{\lambda_k}{y_k}\bigr)$
  is in $\bK_{\tau((C' \times \{0\}^{n-k})\starkn\{0\}^n)}\langle\langle
  X_{n\setminus k},Y_k\rangle\rangle$, and thus in $\bK_{-C'}\langle\langle
  Y_k\rangle\rangle$, hence \eqref{eq:spec-ppart-as-res-plain-phi} and the
  announced expansion.

\item Secondly, Lemma~\ref{thm:spec-hadam-as-res}\,(iii) can be used with
  $C_1=C$, $C_2=C'$, $f=\phi$, $g=\psi$, $\alpha_i=1$, $\beta_i=\lambda_i$, without
  restriction on~$\lambda_i\in \bK$ as $\psi$~involves no monomial with
  negative exponent.  Equation~\eqref{eq:substd-semi-hadam-k-n-as-res} then
  reads
  \begin{equation*}
  \bigl([x_1^>\cdots x_k^>]\phi\bigr)\Bigr|_{x_1=\lambda_1,\dots,x_k=\lambda_k} =
    \Res_{y_1,\dots,y_k} \frac{1}{y_1\cdots y_k}
    \phi\left(\frac{1}{y_1},\dots,\frac{1}{y_k},x_{k+1},\dots,x_n\right)
  \psi(\lambda_1y_1,\dots,\lambda_ky_k),
  \end{equation*}
  where $\psi(\lambda_1y_1,\dots,\lambda_ky_k)$ is in $\bK_{\tau(0\starnk (C'
    \times \{0\}^{n-k}))}\langle\langle X_{n\setminus
    k},Y_k\rangle\rangle$, and thus in $\bK_{C'}\langle\langle
  Y_k\rangle\rangle$, hence \eqref{eq:spec-ppart-as-res-substd-phi} and the
  announced expansion.

\end{itemize}
\end{proof}

The following variant formulation of Theorem~\ref{thm:spec-ppart-as-res}
avoids to potentially get a tautology when some of the specialization point
is zero.

\begin{thm}\label{thm:spec-ppart-as-res'}
For every line-free cone $C\subseteq\set R^n$ that is in opposition
to~$\Rplus^n$ with respect to the first $k$ coordinates, for every $\phi\in
\bK_C\langle\langle X_n\rangle\rangle$, and for every
$\lambda_1,\dots,\lambda_k\in \bK$, we have:
\begin{align}
\biggl(\frac1{x_1\dots x_k} & [x_1^>\cdots x_k^>] \phi\biggr)\biggr|_{x_1=\lambda_1,\dots,x_k=\lambda_k} \nonumber \\
{} &= \label{eq:spec-ppart-as-res-plain-phi'}
  \Res_{y_1,\dots,y_k}
    \frac1{y_1\cdots y_k}
    \left[\frac{\frac1{y_1\cdots y_k}}{(1-\frac{\lambda_1}{y_1})\cdots(1-\frac{\lambda_k}{y_k})}\right]_{\Rminus^k}
    \phi(y_1,\dots,y_k,x_{k+1},\dots,x_n) \\
{} &= \label{eq:spec-ppart-as-res-substd-phi'}
  \Res_{y_1,\dots,y_k}
    \frac1{y_1\cdots y_k}
    \phi\left(\frac1{y_1},\dots,\frac1{y_k},x_{k+1},\dots,x_n\right)
    \left[\frac{y_1\cdots y_k}{(1-\lambda_1y_1)\cdots(1-\lambda_ky_k)}\right]_{\Rplus^k} ,
\end{align}
where the brackets around rational functions in
\eqref{eq:spec-ppart-as-res-plain-phi'}
and~\eqref{eq:spec-ppart-as-res-substd-phi'} denote taking their expansions
in, respectively, $\bK_{\Rminus^k}\langle\langle Y_k\rangle\rangle$ and
$\bK_{\Rplus^k}\langle\langle Y_k\rangle\rangle$.
\end{thm}

\begin{proof}
In a way very similar to the proof of Theorem~\ref{thm:spec-ppart-as-res},
fix $C'$~to $\Rplus^k$ and $\tilde\psi\in \bK_{C'}\langle\langle
X_k\rangle\rangle$ to the geometric series expansion of $\frac 1
{(1-x_1)\cdots(1-x_k)}$.  Introduce as well $\tilde\phi = \phi/(x_1\dots
x_k)$.  The desired positive part can then be represented in two ways as
variants of Hadamard products:
\begin{equation*}
\biggl(\frac1{x_1\dots x_k} [x_1^>\cdots x_k^>]\phi\biggr)\biggr|_{x_1=\lambda_1,\dots,x_k=\lambda_k} =
  \bigl(\tilde\psi \odotkn \tilde\phi\bigr)\bigr|_{x_1=\lambda_1,\dots,x_k=\lambda_k} =
  \bigl(\tilde\phi \odotnk \tilde\psi\bigr)\bigr|_{x_1=\lambda_1,\dots,x_k=\lambda_k} .
\end{equation*}
The proof now follows the same lines as for
Theorem~\ref{thm:spec-ppart-as-res}, using
Lemma~\ref{thm:spec-hadam-as-res} twice: Equations
\eqref{eq:spec-ppart-as-res-plain-phi'}
and~\eqref{eq:spec-ppart-as-res-substd-phi'} follow after observing that
all specializations of any~$\lambda_i$ (in particular, to~0) are well
defined and that a factor $y_1\dots y_k$ can move freely into or out of a
bracket.
\end{proof}

\subsection{Generating series of walks as residues}
\label{sec:gf-as-res}

Now, we want to apply Theorem~\ref{thm:spec-ppart-as-res'} to find
representations as residues of the specializations of a walk
series~$Q(x,y;t)$ at $x=\alpha$ and~$y=\beta$ for numbers $\alpha$ and~$\beta$:
we need this in particular for $(\alpha,\beta) \in \{0,1\}^2$.
For technical reasons, below we will in fact need such representations
for more general $\alpha$ and~$\beta$,
namely elements of a field~$\bK$ of characteristic zero
whose elements have zero derivative with respect to $x$ and~$y$.

We will obtain a formula of the form
\begin{multline}\label{eq:spec-walk-series-as-res}
Q(\alpha,\beta;t) = \left(\frac1{xy}[x^>y^>]\phi\right)\biggr|_{x=\alpha,y=\beta} = \\
\Res_{x,y}
  \frac1{xy}
  \left[\frac{\bar x\bar y}{(1-\alpha\bar x)(1-\beta\bar y)}\right]_{\Rminus^2}
  \phi(x,y;t) =
\Res_{x,y}
  \frac1{xy}
    \phi(\bar x,\bar y;t)
    \left[\frac{xy}{(1-\alpha x)(1-\beta y)}\right]_{\Rplus^2} ,
\end{multline}
for some series~$\phi$ to be determined
and for residues that are residues of rational functions in~$\bK(x,y,t)$.
Our first step is to
express~$Q(x,y;t)$ as a positive-part extraction.  To this end,
Bousquet-Mélou and Mishna provide an appealing formula \cite[Prop.~8]{BMM},
\begin{equation}\label{eq:walk-series-as-mbm-ppart}
xy Q(x,y;t) = [x^>] [y^>] R(x,y;t) ,
\end{equation}
which represents the walk series as an extraction from the \emph{rational
  function\/}~$R(x,y;t) = N(x,y)/(1-tS(x,y))$.  Thus, their formula
requires some care to be used in our context: we know that rational
functions potentially admit several (distinct) positive parts in $\set Q[[x,y;t]]$,
depending on the cone used to expand it in~$\set Q^{\set Z^3}$.
So we need to determine a cone that results in the proper
combinatorial representation of~$Q$.
Additionally, instead of our direct positive-part extraction~$[x^>y^>]$
over~$\set Q^{\set Z^3}$
(and more generally~$\bK^{\set Z^3}$ after we specialize
at $x = \alpha$ and~$y = \beta$),
they use iterated positive-part extractions:
a first that operates coefficient-wise
on series in~$t$, namely $[y^>] : \set Q(x)[y,\bar y][[t]] \rightarrow \set
Q(x)[y][[t]]$; a second that operates coefficient-wise on series in $y$
and~$t$, namely $[x^>] : \set Q[x,\bar x][y][[t]] \rightarrow \set
Q[x,y][[t]]$.  Then, they make the crucial observation that the
(unambiguously defined) coefficients of $[y^>]R$ with respect to $y$
and~$t$ are not only in~$\set Q(x)$, but in fact for the set of~$R$ under
consideration more specifically in~$\set Q[x,\bar x]$, so that $[x^>]$~can
be applied.  At this point, it will be sufficient for us to determine a
cone~$C$ such that
\begin{equation}\label{eq:our-ppart-as-mbm-part}
[x^>y^>] \phi = [x^>] [y^>] R(x,y;t)
\end{equation}
for the series expansion~$\phi$ of~$R$ in $\set Q_C\langle\langle x,y,t\rangle\rangle$.

The cone~$C$ that will do is the cone~$\Gamma$ generated by the vectors
$(1,1,1)$, $(1,-1,1)$, $(-1,0,0)$, so that
elements of the ring~$\set Q_\Gamma[[x,y;t]]$
can only involve monomials $x^k y^m t^n$
for $k,m,n \in \set N$ satisfying $-n \leq m \leq n$ and~$k \leq n$.
This cone is
line-free, making $S_\Gamma := \set Q_\Gamma\langle\langle
x,y;t\rangle\rangle$ a well-defined ring.  Now, for each step set in
Table~\ref{tab:steps}, the product~$tS(x,y)$ involves only monomials
$x^ky^mt$ satisfying $-1 \leq k,m \leq 1$, making the rational
function~$1/\bigl(1-tS(x,y)\bigr)$ admit the expansion
$\sum_{n\geq0}S(x,y)^n t^n$ in~$S_\Gamma$.  In addition, the coefficients
$1/(x+\bar x)$ and~$1/(x+1+\bar x)$ that occur in~$N$ for some of the step
sets admit expansions in~$S_\Gamma$, so that by ring operations, the
rational function~$R$ admits an expansion in~$S_\Gamma$, henceforth
denoted~$\phi = [R]_\Gamma$.  Set $\Gamma' = \Rminus \times \{0\}$, another
line-free cone, so that we can introduce the ring $S_{\Gamma'} = \set
Q_{\Gamma'}\langle\langle x,y\rangle\rangle = \set Q[[\bar x]][x,y,\bar
  y]$.  As vector spaces, we have $S_\Gamma \subset \sum_{j\in\set Z}
S_{\Gamma'} t^j$.  With this observation, $\phi = \sum_{j\in\set N}
[NS^j]_{\Gamma'} t^j$, where the brackets are taken in~$S_{\Gamma'}$.
Therefore, our $[x^>y^>]\phi$ equals Bousquet-Mélou and Mishna's
$[x^>][y^>]R$ if and only if for any~$j\geq0$, our
$[x^>y^>][NS^j]_{\Gamma'}$ equals their $[x^>][y^>](NS^j)$.

We now fix~$j\geq0$.  There are two cases, depending on the step set in
Table~\ref{tab:steps}:

\begin{itemize}

\item\emph{Cases 1 to~16.}  The assumptions of
  Lemma~\ref{lem:same-pos-parts} below are satisfied.  The lemma proves
  that equality holds.

\item\emph{Cases 17 to~19.}  Both $N$ and~$S$ are finite sums in~$\set
  Q[x,\bar x,y,\bar y]$, so $[NS^j]_{\Gamma'}$~is just~$NS^j$, and the
  equality $[x^>y^>][NS^j]_{\Gamma'} = [x^>][y^>](NS^j)$ holds.

\end{itemize}

\begin{lem}\label{lem:same-pos-parts}
Let $N(x,y) = N_1(x) y + N_{-1}(x) \bar{y}$ and $S(x,y) = A_1(x) y + A_0(x)
+ A_{-1}(x) \bar{y}$ be given by Laurent polynomials $N_1$, $A_1$, $A_0$,
and~$A_{-1}$ of\/~$\set Q[x,\bar{x}]$, and by a rational function~$N_{-1}$
of\/~$\set Q(x)$.  Under the additional assumption that the denominator
of~$N_{-1}$ divides~$A_1$ in\/~$\set Q[x,\bar{x}]$, then for all~$j\in\mathbb N$,
$[y^>]NS^j \in \set Q[x,\bar{x},y]$ and $[x^>y^>] [NS^j]_{\Gamma'} =
[x^>][y^>](NS^j)$, where the brackets on the left-hand side are taken with
respect to the cone\/ $\Gamma' = \Rminus \times \{0\}$ and the parentheses on
the right-hand side are in\/~$\set Q(x)[y,\bar y]$.
\end{lem}

\begin{proof}
Write $S = A_1y + \tilde A$ where~$\tilde A \in \set Q[x,\bar{x},\bar{y}]$ and
set $P = N_{-1}A_1 \in \set Q[x,\bar x]$.  Writing brackets for expansions
in~$\set Q_{\Gamma'}\langle\langle x,y\rangle\rangle$, we observe that
$[N]_{\Gamma'} = N_1y+[N_{-1}]_{\Gamma'}\bar y$, $[S]_{\Gamma'} = S$, and
$[N_{-1}]_{\Gamma'}A_1$~is in~$\set Q[x,\bar x]$, implying
$[N_{-1}]_{\Gamma'}A_1 = N_{-1}A_1 = P$.  Given~$j\in\mathbb N$, a
computation yields
\begin{equation*}
NS^j = (N_1y + N_{-1}\bar{y}) (A_1 y + \tilde A)^j =
\left(
  N_1y (A_1 y + \tilde A)^j +
  P \sum_{\ell=1}^j \binom j\ell A_1^{\ell-1} y^{\ell-1} {\tilde A}^{j-\ell}
\right) + N_{-1} \bar{y} {\tilde A}^j .
\end{equation*}
The large parenthesis in the right-hand side is
in~$\set Q[x,\bar{x},y,\bar{y}]$; we call it~$L$.  The last term~$N_{-1} \bar{y}
{\tilde A}^j$ is a polynomial in~$\bar{y}$ with no non-negative powers
of~$y$, and it is thus in~$\set Q(x)[\bar{y}]$.  Consequently, $[x^>][y^>](NS^j)
= [x^>][y^>]L$.  A similar computation delivers $[NS^j]_{\Gamma'} = L +
[N_{-1}]_{\Gamma'} \bar{y} {\tilde A}^j$, where the Laurent polynomial~$L$
is viewed in~$\set Q_{\Gamma'}\langle\langle x,t\rangle\rangle$.  Applying our
operator~$[x^>y^>]$ kills the last term, returning
$[x^>y^>][NS^j]_{\Gamma'} = [x^>y^>]L = [x^>] [y^>] L$.  The result is
proved.
\end{proof}

The following lemma summarizes our discussion.

\begin{lem}\label{lem:Q-as-res-res}
For any field\/~$\bK$ of characteristic zero
whose elements have zero derivative with respect to $x$ and~$y$,
for any~$(\alpha,\beta)\in\bK^2$,
and for any step set in Table~\ref{tab:steps},
Equation~\eqref{eq:spec-walk-series-as-res}~above holds when $\phi$~is set
to the expansion of~$R = N/(1-tS)$ in~$\set Q_\Gamma\langle\langle
x,y;t\rangle\rangle$, for the cone~$\Gamma$ generated by the vectors
$(1,1,1)$, $(1,-1,1)$, and~$(-1,0,0)$.
\end{lem}

\begin{proof}
Set $\bK, \alpha$, $\beta$, $\Gamma$, and~$\phi$ as in the statement.
The first identity in~\eqref{eq:spec-walk-series-as-res} follows
from \eqref{eq:walk-series-as-mbm-ppart} and~\eqref{eq:our-ppart-as-mbm-part};
the other identities are proven by Theorem~\ref{thm:spec-ppart-as-res'},
used with the field~$\bK$ of the present lemma,
provided we prove that $\Gamma$~is
in opposition to $\Rplus^3$ with respect to the first two variables.
To see this, note: \emph{(i)\/}~that both
$\pi_2(\Gamma)$ and~$\pi_2(\Rplus^3)$ equal~$\Rplus$ and are thus line-free
cones; \emph{(ii)\/}~that $\Gamma\cap\Rplus^3$ is generated by $(1,1,1)$,
$(1,0,1)$, $(0,1,1)$, and~$(0,0,1)$, so that its intersection with~$\set
R^2\times\{0\}$ reduces to~$\{(0,0,0)\}$.
\end{proof}

\subsection{Annihilating generating series of walks by creative telescoping}
\label{sec:ann-by-ct}

Our general goal is to justify via creative telescoping the computation of a
differential equation in~$t$ satisfied by the residue with respect to $x$
and~$y$. The objects manipulated during the computation are rational functions,
and it is usually taken for granted that results of these calculations can be
ported without difficulties to other domains. Here, we want to formally prove
this for the present situation.
For this section again, $\bK$~is a field satisfying the conditions
of Lemma~\ref{lem:Q-as-res-res}.

\begin{lem}\label{lem:resF-is-0}
Let $F \in \bK(x,y,t)$.  If there exist $U,V\in \bK(x,y,t)$ such that
$F=\partial_xU+\partial_yV$, then for any cone~$C$ for which $F,U,V$ all
admit an expansion, $\Res_{x,y} [F]_C = 0$.
\end{lem}

\begin{proof}
From the hypothesis, it follows $[F]_C=\partial_x[U]_C+\partial_y[V]_C$. As
the residue of a derivative is zero by Lemma~\ref{lem:res-partial}, the
result is proved by linearity.
\end{proof}

As in our applications the creative-telescoping certificates $U$ and~$V$
will be large and messy expressions, whose manipulation can be costly, it
is comfortable to state a refinement of the previous lemma that avoids the
need to inspect $U$ and~$V$ so as to determine a cone~$C$ for which they
admit expansions.

\begin{lem}\label{lem:resF-is-0'}
Let $F \in \bK(x,y,t)$.  If there exist $U,V\in \bK(x,y,t)$ such that
$F=\partial_xU+\partial_yV$, then for any line-free cone~$C$ for which $F$~admits an
expansion, $\Res_{x,y} [F]_C = 0$.
\end{lem}

\begin{proof}
Since $C$ is line free, there exists a monomial order~$\preccurlyeq$ such that $C$~is
compatible with~$\preccurlyeq$.  As $F$, $U$ and~$V$ now admit expansions
in the field~$\bK_\preccurlyeq((x,y,t))$, there is a cone~$C'$ containing~$C$
so that all three rational functions admit expansions
in~$\bK_{C'}\langle\langle x,y;t\rangle\rangle$.  Finally applying
Lemma~\ref{lem:resF-is-0} proves the result.
\end{proof}

\begin{thm} \label{thm:P-by-CT}
For some fixed step set $\cS$, let $N(x,y)$ and $S(x,y)$ be as in
Table~\ref{tab:steps}, let $R(x,y;t)$ denote $N(x,y)/(1-t S(x,y)) \in
\set Q(x,y;t)$, and let $Q(x,y;t)=\sum_{n,i,j=0}^\infty q_{i,j;n}x^i y^j t^n \in
\set Q[x,y][[t]]$ be the corresponding generating function of walks.  Let
$\Gamma$ be the cone generated by the vectors $(1,1,1)$, $(1,-1,1)$,
and~$(-1,0,0)$, so that the expansion~$[R]_\Gamma$ exists in
$\set Q_\Gamma\langle\langle x,y;t\rangle\rangle$ and $Q$~can be identified with
its expansion there.
For\/ $\bK$ a field containing~$\set Q$,
whose elements have zero derivative with respect to $x$ and~$y$,
let $(\alpha,\beta) \in \bK^2$ and set $F$ to either of
\begin{equation*}
F_1 = \frac1{xy} \frac{R(x,y;t)}{(\alpha-x)(\beta-y)}
\qquad\text{and}\qquad
F_2 = \frac{R(1/x,1/y;t)}{(1-\alpha x)(1-\beta y)} .
\end{equation*}
Suppose $L\in \bK[t]\langle\partial_t\rangle$ and $F,U,V\in \bK(x,y,t)$ are
such that $L(F)=\partial_xU+\partial_yV$.  Then $L(Q(\alpha,\beta;t)) = 0$.
\end{thm}

\begin{proof}
The line-free cones $C_1 = \tau\bigl((\Rplus^2\times\{0\}) \underset{2\leq
  3}{\star} \Gamma\bigr)$ and $C_2 = \tau\bigl(\Gamma \underset{3\geq
  2}{\star} (\Rplus^2\times\{0\})\bigr)$ are, by the construction of
Lemma~\ref{lem:Q-as-res-res}, the cones such that the product in the
left-hand residue in~\eqref{eq:spec-walk-series-as-res} takes place in
$\bK_{C_1}\langle\langle t,x,y\rangle\rangle$ (with indeterminates in
this order) while the product in the right-hand residue takes place in
$\bK_{C_2}\langle\langle t,x,y\rangle\rangle$ (same order).
With this notation, Equation~\eqref{eq:spec-walk-series-as-res} rewrites
\begin{equation}\label{eq:Q-as-res-Fi}
Q(\alpha,\beta;t) = \Res_{x,y} [F_1]_{C_1} = \Res_{x,y} [F_2]_{C_2} .
\end{equation}
Set~$i$ to either of 1 and~2 so that~$F = F_i$, then set~$C$ to~$C_i$.  By
Lemma~\ref{lem:resF-is-0} and the expression of~$L(F)$ in terms of $U$
and~$V$, we have $\Res_{x,y} L([F]_C) = \Res_{x,y} [L(F)]_C = 0$.  Since
$L$~is free of~$x,y$, we have $\Res_{x,y} L([F]_C) = L(\Res_{x,y} [F]_C)$.
By~\eqref{eq:Q-as-res-Fi}, we have $L(Q(\alpha,\beta;t)) = L(\Res_{x,y} [F]_C) = 0$,
as claimed.
\end{proof}

\section{Results}

\subsection{Obtaining proved hypergeometric formulas}\label{sec:pFqSolutions}

Let $\bK$ be a field of characteristic zero
whose elements have zero derivative with respect to $x$ and~$y$.
To obtain a proven expression for the specialization~$Q(\alpha,\beta;t)$
of~$Q(x,y;t)$ at~$(x,y) = (\alpha,\beta) \in \bK^2$,
it is sufficient to find:
\begin{enumerate}
\item a nonzero operator $L \in \bK[t]\langle\partial_t\rangle$
  for which $L(Q(\alpha,\beta;t))$ is proved to be~0,
\item a closed-form solution~$f(t)$ of~$L$
  that agrees with~$Q(\alpha,\beta;t)$ to a sufficiently high order.
\end{enumerate}

Concerning point~1,
Theorem~\ref{thm:P-by-CT} provides such a~$L$
by applying creative telescoping.
The rational functions $U$ and~$V$ that are also part of the output
can be discarded,
as they are not needed for our application to positive-part extraction.
Taking~$\bK = \set Q$,
this is already sufficient to obtain formally proved differential equations
for the generating functions~$Q(t) = Q(1,1;t)$,
which was the initial goal of our work.

To obtain a proven equation for~$Q(x,y;t)$, one could in principle
follow the same strategy:
introduce the bivariate function field~$\bK = \set Q(\alpha,\beta)$,
for indeterminates $\alpha$ and~$\beta$;
use Theorem~\ref{thm:P-by-CT} for this~$\bK$;
then replace $\alpha$ and~$\beta$ by $x$ and~$y$ at the end.
However, this computation would be too slow to work in practice.
One could instead consider computing operators
$L_{\alpha,\beta} \in \set Q[t]\langle\partial_t\rangle$
for sufficiently many points $(\alpha,\beta)\in \set Q^2$
and then interpolate into $L \in \set Q(x,y)[t]\langle\partial_t\rangle$.
This would raise the delicate problem of determining
how many interpolation points need to be taken
to prove that $L$~is correct.

What does work in practice is to refrain from computing an annihilator~$L$
for the complete generating function~$Q(x,y;t)$.
Instead, we obtain two formally proved differential equations:
a first one for~$Q(x,0;t)$,
by using creative telescoping over~$\bK = \set Q(\alpha)$
and substituting $x$ for~$\alpha$;
a second one for~$Q(0,y;t)$,
by using creative telescoping over~$\bK = \set Q(\beta)$
and substituting $y$ for~$\beta$.
In either case, the ground field~$\bK$ to be used is
a univariate rational-function field,
which permits calculations to terminate in reasonable time.
Although these operators do not lead to an annihilator for~$Q(x,y;t)$,
they are used below to get a closed form for it.

For all instances of creative telescoping,
we effectively used Koutschan's package%
\footnote{\url{http://www.risc.jku.at/research/combinat/software/ergosum/RISC/HolonomicFunctions.html}}
 \verb+HolonomicFunctions+ for Mathematica,
then double-checked with Chyzak and Pech's package%
\footnote{\url{http://algo.inria.fr/chyzak/mgfun.html}}
 \verb+Mgfun+ for Maple.

Concerning point~2, the operators obtained by point~1 all factor miraculously
into factors that have order~1
with the exception of the left-most one that can have order 1 or~2,
as we could verify using Maple's \verb+DEtools[DFactor]+.
At this point, we can capitalize on the work in~\cite{FaHo11}
and obtain $\pFqnoargs21$ formulas for $Q(1,1;t)$, $Q(x,0;t)$, and~$Q(0,y;t)$
by the method already used in~\cite{BoChHoPe11}.
To this end, we appealed to Imamoglu’s implementation%
\footnote{\url{http://www.math.fsu.edu/~eimamogl/hypergeometricsols}}
 for Maple.

Finally, the kernel equation
(first equation in Lemma~4 of~\cite{BMM}, which we reproduce here)
\begin{equation*}
xy(1-tS(x,y))Q(x,y;t) =
  xy - tx A_{-1}(x) Q(x,0;t) - ty B_{-1}(y) Q(0,y;t)
  + \epsilon t Q(0,0;t) ,
\end{equation*}
where $A_{-1}$ and~$B_{-1}$ are defined as in Lemma~\ref{lem:same-pos-parts}
and $\epsilon$~is~1 if and only if $(-1,-1) \in \cS$ and 0~otherwise,
expresses~$Q(x,y;t)$ in terms of the obtained formulas
for $Q(x,0;t)$ and $Q(0,y;t)$.
As observed in the introduction,
all $\pFqnoargs21$ functions appearing in formulas
can be expressed in terms of complete elliptic integrals.
Furthermore, for each of the 19 step sets, those functions do not depend
on the choice of specialization points $\alpha$ and~$\beta$,
so that $Q(x,y;t)$~involves a single hypergeometric function
(up to variations under contiguity and derivatives).

We remark that
the nature of the generating functions for our walk models is even more rigid:
not only can each model be expressed
in terms of a single $\pFqnoargs21$ function,
but, if one can afford increasing degrees and/or algebraic extensions,
all models can be expressed in terms of the same single $\pFqnoargs21$ function,
namely $\pFq21{\frac1{12},\frac5{12}}{1}{u}$.
The emergence of this function is no surprise,
owing to Takeuchi's classification~\cite{takeuchi}.
This classification establishes connected components
in the class of $\pFqnoargs21$ functions,
under simple kinds of transformations.
One of the connected components is well represented as the following diagram
(see \cite[Table~1, `Classical transformations']{vidunas}
or \cite[Section~4, Diagram~(1)]{takeuchi}):

\bigskip
\centerline{%
\begin{tikzpicture}[scale=1.5]
\node (F-12-13) {$(\frac12,\frac13;1)$};
\node (F-13-23) at (2,0) {\fbox{$(\frac13,\frac23;1)$}};
\node (F-14-34) at (4,0) {\fbox{$(\frac14,\frac34;1)$}};
\node (F-12-14) at (6,0) {$(\frac12,\frac14;1)$};
\node (F-16-13) at (1,1) {$(\frac16,\frac13;1)$};
\node (F-112-512) at (3,1) {$(\frac1{12},\frac5{12};1)$};
\node (F-18-38) at (5,1) {$(\frac18,\frac38;1)$};
\node (F-12-16) at (3,-1) {$(\frac12,\frac16;1)$};
\node (F-12-12) at (4,-2) {\fbox{$(\frac12,\frac12;1)$}};
\draw[-] (F-12-13)
  to node[above left] {2} (F-16-13)
  to node[above right] {2} (F-13-23)
  to node[above left] {4} (F-112-512)
  to node[above right] {3} (F-14-34)
  to node[above left] {2} (F-18-38)
  to node[above right] {2} (F-12-14);
\draw[-] (F-112-512)
  to node[right] {2} (F-12-16)
  to node[above right] {3} (F-12-12)
  to node[right] {2} (F-14-34);
\end{tikzpicture}
}
\bigskip

\noindent
This diagram should be read as follows:
if $(a,b;c)$ and $(a',b';c')$ are the endpoints of an edge labelled~$d$,
with the latter endpoint above the former in the diagram, then
\begin{equation*}
\pFq21{a,b}{c}{t} = \sqrt[m]{r(t)} \cdot \pFq21{a',b'}{c'}{w(t)}
\end{equation*}
for some positive integer~$m$, a rational function~$w(t)$ of degree~$d$,
and another rational function~$r(t)$.

Up to integer shifts,
all entries in our table are written in terms of $\pFqnoargs21$~functions
involving the boxed parameters of our diagram.
If we do not mind increasing the degree of~$w$ by a factor 3, 4, or~6,
they can all be rewritten in terms of $\pFq21{\frac1{12},\frac5{12}}{1}{w}$,
itself related to elliptic curves and modular forms
(see, e.g., \cite{StienstraBeukers-1985-PFE} or~\cite{Stiller-1988-CAF}).
Alternatively, all of our formulas can be expressed up to algebraic extensions
in terms of $\pFq21{\frac12,\frac12}{1}{w}$,
itself related to the complete elliptic function~$K(k)$.

The connections mentioned above can be made explicit
by providing formulas for the main edges in Takeuchi's diagram:
\begin{itemize}
\item The hypergeometric series
  $\pFq21{\frac12,\frac12}{1}{w}$ and $\pFq21{\frac14,\frac34}{1}{w}$
  are related by the well-known duplication
  formula~\cite[Eq.~(3.1.7)]{AndrewsAskeyRoy-1999-SF}:
  \begin{equation*}
  \left(1-\frac u2\right)^{1/2} \pFq21{\frac12,\frac12}{1}{u} =
    \pFq21{\frac14,\frac34}{1}{\left(\frac u{2-u}\right)^2} .
  \end{equation*}
\item Less well-known formulas \cite[Eq.\ (119) \&
  (126)]{Goursat-1881-EDL}
  result, upon evaluating at~$\eta=1/12$, into
  \begin{align*}
  (1+3u)^{1/4} \pFq21{\frac14,\frac34}{1}{u} &=
    \pFq21{\frac1{12},\frac5{12}}{1}{\frac{27u(1-u)^2}{(1+3u)^3}} , \\
  \left(1-\frac{8v}9\right)^{1/4} \pFq21{\frac13,\frac23}{1}{v} &=
    \pFq21{\frac1{12},\frac5{12}}{1}{\frac{64v^3(1-v)}{(9-8v)^3}} .
  \end{align*}
  This relates $\pFq21{\frac14,\frac34}{1}{t}$ and
  $\pFq21{\frac13,\frac23}{1}{t}$ via a degree-4 algebraic transformation.
\end{itemize}
Combining these formulas shows that
$\pFq21{\frac13,\frac23}{1}{t}$ can be expressed in terms of
$\pFq21{\frac12,\frac12}{1}{t}$
via a degree-4 function in a degree-6 algebraic extension, as indicated in the diagram.

\subsection{Algebraic nature of the counting series} \label{sec:alg-or-trans}

To prove the transcendence of one of our trivariate counting series~$Q(x,y;t)$, it is
sufficient to prove the transcendence of one of its evaluations for specific
values $\alpha$ of~$x$ and $\beta$ of~$y$---and indeed, the complete generating
functions for all cases listed in Table~\ref{tab:steps} are transcendental as
a consequence of $Q(0,0;t)$ being transcendental.
The same holds true for the bivariate series $Q(x,0;t)$ and~$Q(0,y;t)$.

We now turn to the case of univariate enumerating series~$Q(\alpha,\beta;t)$,
that is, specializations of~$Q(x,y;t)$
at specific numerical values $\alpha$ and~$\beta$ for $x$ and~$y$,
and we describe how we prove transcendence of a given~$Q(\alpha,\beta;t)$
after having computed an annihilating operator~$L$ for it.
Remember from the discussion in
Section~\ref{sec:pFqSolutions} that $L$~factors in the form $L'L''$, where
$L'$~has order 1 or~2. Define $\tilde Q(t) = L''(Q(\alpha,\beta;t))$, which has to be
algebraic if $Q(\alpha,\beta;t)$~is algebraic. In each choice of a step set and of
$\alpha,\beta$ in~$\{0,1\}$, the computation of a few terms of the series expansion
of~$Q$ proves, upon applying~$L''$, that $\tilde Q$~is non-zero. Now, if
$L'$~has order~2, Kovacic's algorithm~\cite{Kovacic-1986-ASS} decides if it has
(non-zero) algebraic solutions: if not, this proves that $Q$~cannot be
algebraic. The case when $L'$~has order~1 is similar, and simpler as solving a
first-order linear ODE is very elementary.
In our calculations,
we used Maple's \verb+DEtools[kovacicsols]+ in the second-order case
and \verb+DEtools[expsols]+ in the first-order case
for this test of existence of algebraic solutions of the operators~$L$.
This proves the first part of
Theorem~\ref{thm:trans}, namely that among the $19 \times 4$ combinatorially
meaningful specializations~$Q(\alpha,\beta;t)$ ($\alpha,\beta\in \{0,1\}$) of the complete
generating function~$Q(x,y;t)$, only four cases could possibly be algebraic
functions: Case~17 for $x$ and~$y$ specialized to~1, and Case~18 for $(x,y)$
specialized to $(0,1)$, $(1,0)$, and~$(1,1)$.

In these four cases, we indeed prove the algebraicity of the corresponding
counting series. The procedure we followed was simply to use specific
algorithms when solving the relevant differential equation $L(y)=0$, and note
that it admits a basis of algebraic solutions.
Explicit expressions in these four cases are then found using the initial terms of the counting series.
For instance, in case 18 for $(\alpha,\beta)=(1,0)$ and for $(\alpha,\beta)=(0,1)$, the operator~$L$ proved to annihilate $Q(1,0;t)$ and $Q(0,1;t)$ is
\[L =
t^3(2t+1)(6t-1)\partial_t^4+4t^2(48t^2+13t-3)\partial_t^3+36t(3t+1)(8t-1)\partial_t^2+(1152t^2+168t-24)\partial_t+288t+24,\]
which admits the basis of solutions
\[\left\{ s_1 = \frac{1}{t}, \; s_2 = \frac{4t^2-8t+1}{t^3}, \; s_3 = \frac{12t^2-1}{t^3}, \; s_4 = \frac{(2t+1)^{1/2} (1-6t)^{3/2}}{t^3} \right\}.\]
Therefore, $Q(1,0;t)$ is equal to a linear combination $c_1s_1 +c_2s_2 +c_3s_3 +c_4s_4$ for some constants $c_1, c_2, c_3, c_4$.
Identification of initial terms using $Q(1,0;t)=1+t+4t^2+12t^3+O(t^4)$ provides a linear system in the~$c_i$'s, finally proving that
 $Q(1,0;t)$ is equal to
$s_4 -s_2$. This concludes the proof of Theorem~\ref{thm:trans}.

Note that, for any of the 19 models, the excursions generating functions
$Q(0,0;t)$ could alternatively be proved transcendental by an argument based
on asymptotics, similar to the one in~\cite{BoKaSa14}: using~\cite{DeWa15} the
coefficient of $t^{12n}$ in $Q(0,0;t)$ grows like $\kappa \rho^n / n^\gamma$ for
$\gamma\in\{3,4,5\}$ (see Table~\ref{tab:kappa-in-F00t}), and this implies
transcendence of $Q(0,0;t)$ by \cite[Theorem~D]{Flajolet87}. By contrast, note
that this asymptotic argument \emph{is not} sufficient to prove the
transcendence of \emph{all} the other transcendental specializations, as
showed for instance by Case 7 at $(0,1)$ ($\gamma=3/2$) and at $(1,1)$
($\gamma=1/2$), and by Case 17 at $(1,0)$ ($\gamma=5/2$), see
Tables~\ref{tab:kappa-in-F01t}, \ref{tab:kappa-in-F10t} and
\ref{tab:kappa-in-F11t}.

\subsection{Asymptotic formulas for coefficients} \label{sec:asymptotics}

This section discusses how to develop asymptotic estimates
on the counting coefficients from the closed forms for the generating series,
in each of the 19 walk models with prescribed length
(and unprescribed endpoint).
This bases on asymptotics-transfer theorems
(developed in Analytic Combinatorics)
to get the results.
Not all asymptotic formulas could be proved by our approach,
so we report here on a reduction of the proofs of asymptotic formulas
to the proof of a number of integral representation of rational constants,
some of which remain conjectural.
What is most difficult is proving the constants ($\kappa$~below)
in front of the asymptotic pattern $\rho^n / n^\gamma$,
which were initially heuristically discovered by numerical calculations
(number recognition)~\cite{BK}.

All of our generating series consist of hypergeometric functions,
rational or algebraic factors, and iterated primitivations.
As was explained in the introduction,
so far all these mathematical ingredients have been viewed
as formal Laurent series in the field~$\set C((t))$
and operators on them,
but the same series are convergent as expansions of functions
meromorphic in disks centered at~0.
This analytic interpretation is crucial
to the method of the present section.

A subtle point in this regard is the interpretation
of the primitivation operator in the context of series of functions.
To match the fact that the (formal) primitive~$\int A$ of a (formal) series
$A(t) = \sum_{n \geq m, \ n \neq -1} a_n t^n$ of~$\set C((t))$
is $(\int A)(t) = \sum_{n \geq m, \ n \neq -1} a_n t^{n+1} / (n+1)$
and has a zero constant term,
we introduce an operator~$\pp$ of \emph{polar part}, defined by
\[ \pp(A)(t) = \sum_{n < 0} a_n t^n , \]
and we interpret~$\int$ analytically by the integral representation
\begin{equation}\label{eq:int-repr}
(\int A)(t) =
  - \int_t^{\infty} \pp(A)(u) \, du
  + \int_0^t (A - \pp(A))(u) \, du .
\end{equation}
This split is needed, as the meromorphic function that expands as~$A$
is in general not integrable at~0.
The integral representation~\eqref{eq:int-repr} will occur implicitly
in the proofs
of Theorems \ref{thm:asympt-case-7} and~\ref{thm:asympt-case-5} below.

In the generating functions, hypergeometric terms
take the form $\pFq21{a,b}{c}{w}$, with good constraints on $a$,
$b$, and~$c$, so that enough of their asymptotic behavior (e.g., at~$w =
1$) can be derived easily, and where $w$~is a rational or algebraic term
in~$t$.  Singularities of the generating series can occur at poles
or branch point of the cofactors, or at values of~$t$ where $w = w(t)$
becomes one of the singularities of the~$\pFqnoargs21$, namely $0$, $1$, and~$\infty$.

We first expected that the problem of proving all coefficient asymptotics
would be just a problem of Analytic Combinatorics (in the sense of the
theory well exposed by Flajolet and Sedgewick
in~\cite[Chap.~VI]{FlajoletSedgewick-2009-AC}, largely basing on transfer
theorems by Flajolet and Odlyzko~\cite{FlajoletOdlyzko-1990-SAG}).  It
turned out that the part of these theories to extract asymptotics is
minimal, most of the difficulties being in computing asymptotic expansions
of non-standard representations of functions.  (These expansions are then
input to transfer theorems.)

What also made the problem more difficult is a broad spectrum of the
possible asymptotic phenomenons/asymptotic patterns (i.e., what
ingredient of the function contributes to the asymptotics).
A recurrent situation is when
the dominant singularity is
caused by an algebraic factor between two primitivation operators.
In some cases, this is worsened
by a polar part of the integrand in the inner primitivation:
this part is necessary and disappears in the expansion
as a formal power series,
owing to the alternation of primitivations
and multiplications by specific algebraic factors;
but to get the constants~$\kappa$ by analytic means,
it is easier to remove this polar part in calculations of asymptotic estimates.

In view of this,
we have in fact formally proved only some of the constants~$\kappa$.
However, in all the cases,
those~$\kappa$---as well as other auxilliary constants that appear in the asymptotic study---%
can be formulated as definite integrals of univariate functions
obtained from~$Q(x,y;t)$.
Proving the asymptotic formulas up to the constants~$\kappa$
is equivalent to proving closed-form evaluations of those integrals.
This is exemplified and described now in four typical cases
that cover all difficulties that may occur.

\subsubsection{Case~4: walks with steps from\/ $\protect\smallstepset11111111$}

\begin{thm}\label{thm:asympt-case-4}
The number of walks of length~$n$ in the quarter plane
that use steps from the step set\/~$\smallstepset11111111$,
start at the origin, and end anywhere
is asymptotically equivalent to $\frac8{3\pi} 8^n$.
\end{thm}

\begin{proof}
For this case, the generating series is proved
by the method of Section~\ref{sec:pFqSolutions}
to be
\begin{equation*}
Q(1,1;t) =
\frac1t \int_0^t \frac1{(1+4u)^3} \,
  {\pFq21{\frac32,\frac32}{2}{\frac{16u(1+u)}{(1+4u)^2}}} \, du .
\end{equation*}
The possible singularities of the integrand in~$Q(1,1;t)$ are
the pole of the rational factor and
the values of~$u$ for which the argument~$w$ of the hypergeometric function
is either 1 or~$\infty$.
These are $u = -1/4$ and~$u = 1/8$,
and the dominant singularity is indeed at~$1/8$, where~$w = 1$.
The asymptotic analysis of the~$\pFqnoargs21$ at~1
(using normalization by~(15.8.12)
in~\cite{NIST-HMF} before using~(15.8.10) in~\cite{NIST-HMF}%
\footnote{Alternative references are \url{http://dlmf.nist.gov/15.8.E12}
  and \url{http://dlmf.nist.gov/15.8.E10}, respectively.}%
) delivers the following asymptotic equivalent
of the coefficient~$c_n$ of~$t^n$ in the integrand:
\begin{equation*}
c_n = [t^n] \, \frac1{(1+4t)^3} \,
  {\pFq21{\frac32\ \frac32}{2}{\frac{16t(1+t)}{(1+4t)^2}}}
\sim \frac8{3\pi} 8^n .
\end{equation*}
Then, making explicit
the action on coefficients of integrating and dividing by~$t$
leads to
\begin{equation*}
[t^n] \, Q(1,1;t) = [t^{n+1}] \sum_m \frac{c_m}{m+1} t^{m+1} = \frac{c_n}{n+1}
  \sim \frac8{3\pi} \frac{8^n}n .
\end{equation*}
This is the announced result, including a proved constant~$\kappa = 8/(3\pi)$.
\end{proof}

\subsubsection{Case~3: walks with steps from\/ $\protect\smallstepset11011101$}

\begin{thm}
The number of walks of length~$n$ in the quarter plane
that use steps from the step set\/~$\smallstepset11011101$,
start at the origin, and end anywhere
is asymptotically equivalent to $\frac{\sqrt6}\pi \frac{6^n}n$.
\end{thm}

\begin{proof}
The generating series is proved
by the method of Section~\ref{sec:pFqSolutions}
to be
\begin{equation*}
Q(1,1;t) =
\frac1t \int_0^t
  \frac{1-2u}{((1+2u)(1+6u))^{3/2}} \,
  \pFq21{\frac32,\frac32}{2}{\frac{16u}{(1+2u)(1+6u)}} \, du .
\end{equation*}
Potential singularities of the integrand are
singularities of the algebraic factor ($u = -1/2$ and~$u = -1/6$)
and points where the argument~$w$ of the hypergeometric function
becomes infinite (same values for~$u$)
or tends to~1 ($u = 1/6$).
There are two dominant singularities:
at~$+1/6$, where~$w = 1$, and at~$-1/6$, where~$w = - \infty$.
Following the principles of Analytic Combinatorics,
asymptotic terms contributed by both points
have to be considered and added.
The same approach as for the proof of Theorem~\ref{thm:asympt-case-4},
replicated in parallel,
delivers two respective contributions
$\frac{\sqrt6}\pi 6^n$ and $\frac{\sqrt6}{4\pi} \frac{6^n}n$.
However, the first one dominates
and is the only one to remain in an asymptotic equivalent
for the coefficient of~$t^n$ in the series expansion of the integrand.
As in the proof of Theorem~\ref{thm:asympt-case-4},
the operator~$\frac1t \int$ translates
in just a division by~$n$.
\end{proof}

\subsubsection{Case~7: walks with steps from\/ $\protect\smallstepset11001001$}

\begin{conj}\label{conj:asympt-case-7}
The number of walks of length~$n$ in the quarter plane
that use steps from the step set\/~$\smallstepset11001001$,
start at the origin, and end anywhere
is asymptotically equivalent to $\frac4{3 \sqrt\pi} \frac{4^n}{\sqrt n}$.
\end{conj}

\begin{thm}\label{thm:asympt-case-7}
Define the real number
\begin{multline*}
I = \int_0^{1/4}
  \left\{
    \frac{(1-4v)^{1/2}(\frac12+v)}{v^2} \left[1+\frac1{2v(1+2v)(1+4v^2)^{1/2}} \times {} \right. \right. \\
    \left. \left. \left((1-v) \, \pFq21{\frac32,\frac12}{1}{\frac{16v^2}{1+4v^2}}
     -(1+v)(1-4v+8v^2) \, \pFq21{\frac12,\frac12}{1}{\frac{16v^2}{(1+4v^2)}}\right)
  \right] - \frac1{v^2} \right\}
  \, dv .
\end{multline*}
If\/ $I = -2$, then Conjecture~\ref{conj:asympt-case-7} holds.
\end{thm}

\begin{proof}
The generating series is proved
by the method of Section~\ref{sec:pFqSolutions}
to be
\begin{multline*}
Q(1,1;t) =
\frac1{t(t-1)} \int_0^t
  \frac{u}{(1-4u)^{3/2}} \left\{4+\int_0^u
    \frac{(1-4v)^{1/2}(\frac12+v)}{v^2} \left[1+\frac1{2v(1+2v)(1+4v^2)^{1/2}} \times {} \right. \right. \\
    \left. \left. \left((1-v) \, \pFq21{\frac32,\frac12}{1}{\frac{16v^2}{1+4v^2}}
     -(1+v)(1-4v+8v^2) \, \pFq21{\frac12,\frac12}{1}{\frac{16v^2}{(1+4v^2)}}\right)
  \right]
  \, dv \right\}
\, du .
\end{multline*}
Here, we have a linear combination of hypergeometric functions,
but this is no problem:
the dominant singularity is at~$t = 1/4$,
and dominates~$\pm\frac1{2\sqrt3}$, where~$w = 16t^2/(1+4t^2) = 1$,
as well as~$\pm\frac{\sqrt{-1}}2$, where~$|w| = \infty$.
Call~$f$ the innermost integrand.
It proves successful to remove from~$f$ its singular part at the origin,
whose contribution after double integration can be obtained
by an easy use of the theory.
Near~0, $f$~behaves indeed as $1/t^2 - 9 + O(t)$.
Proceeding by linearity after defining $g(t) = f(t) - 1/t^2$,
which is analytic on $[0,1/4]$,
we get
\begin{equation*}
Q(1,1;t) =
\frac{\sqrt{1-4t}}{2t(t-1)} +
\frac1{t(t-1)} \int_0^t \frac{u \int_0^u g(v) \, dv}{(1-4u)^{3/2}} \, du .
\end{equation*}
Near~$1/4$, $g(v)$~behaves as $- 16 + \Theta(\sqrt{1-4v})$,
so that $\int_v^{1/4} g(v) \, dv = -4(1-4v) + \Theta((1-4v)^{3/2})$.
As $I = \int_0^{1/4} g(v) \, dv$,
we have $u \int_0^u g(v) \, dv = I/4 + \Theta(1-4u)$.
Next, integrating with respect to~$u$ yields
\begin{equation*}
\int_0^t \frac{u \int_0^u g(v) \, dv}{(1-4u)^{3/2}} \, du =
  \frac{I}{8 \sqrt{1-4t}} + O(1) ,
\qquad\text{so that}\qquad
Q(1,1;t) = \frac{-2 I}{3 \sqrt{1-4t}} + O(1) .
\end{equation*}
Finally, by a transfer theorem,
the number of walks of length~$n$ in Conjecture~\ref{conj:asympt-case-7}
is proved to be
asymptotically equivalent to $\frac{-2 I}{3 \sqrt\pi} \frac{4^n}{\sqrt n}$
provided~$I \neq 0$,
and the announced asymptotic formula of the conjecture
holds if and only if the constant identity~$I = -2$ holds.
\end{proof}

Note that the constant identity~$I = -2$ can be checked
numerically (for instance up to 100~digits).
In particular, proving~$I \neq 0$ would be sufficient to prove
that the wanted asymptotic expansion
is of the form~$\kappa \frac{4^n}{\sqrt n}$.
In this regard, a plot of~$g$ suggests that $g(t) < -7$ for~$t \in [0,1/4]$,
therefore that~$I < - 7/4$.

\subsubsection{Case~5: walks with steps from\/ $\protect\smallstepset01001001$}

\begin{conj}\label{conj:asympt-case-5}
The number of walks of length~$n$ in the quarter plane
that use steps from the step set\/~$\smallstepset01001001$,
start at the origin, and end anywhere
is asymptotically equivalent to $\frac12 \sqrt{\frac3\pi} \frac{3^n}{\sqrt n}$.
\end{conj}

\begin{thm}\label{thm:asympt-case-5}
Define the real number
\begin{multline*}
I = \int_0^{1/3}
  \left\{
    \frac{(1-3v)^{1/2}}{v^3(1+v)^{1/2}} \times {} \right. \\
      \left. \left[ 1+(1-10v^3) \, \pFq21{\frac34,\frac54}{1}{64v^4} +
             6v^3(3-8v+14v^2) \, \pFq21{\frac54,\frac74}{2}{64v^4} \right]
- \frac2{v^3} + \frac4{v^2} \right\}
    \, dv .
\end{multline*}
If\/ $I = 1$, then Conjecture~\ref{conj:asympt-case-5} holds.
\end{thm}

\begin{proof}
The generating series is proved
by the method of Section~\ref{sec:pFqSolutions}
to be
\begin{multline*}
Q(1,1;t) =
\frac1{t(t-1)} \int_0^t
  \frac{u^2}{(1+u)^{1/2}(1-3u)^{3/2}} \left\{-7+ \int_0^u
    \frac{(1-3v)^{1/2}}{v^3(1+v)^{1/2}} \times {} \right. \\
      \left. \left[ 1+(1-10v^3) \, \pFq21{\frac34,\frac54}{1}{64v^4} +
             6v^3(3-8v+14v^2) \, \pFq21{\frac54,\frac74}{2}{64v^4} \right]
    \, dv \right\}
\, du .
\end{multline*}
The situation is similar to Case~7,
starting with a dominant singularity at~$t = 1/3$,
close to but below~$\frac1{\sqrt8}$, where~$w = 64t^4 = 1$.
In this case,
removing the polar part of the innermost integrand~$f$ at the origin,
namely $\frac2{t^3} - \frac4{t^2}$,
is again enough to find an equivalent formulation
of the result as a numerical integral.
This time,
$g(t) = f(t) - \frac2{t^3} + \frac4{t^2} = - 18  + \Theta(\sqrt{1-3v})$
near~$t = 1/3$,
so that $\int_v^{1/3} g(v) \, dv = -6(1-3v) + \Theta((1-3v)^{3/2})$.
Continuing by linearity and using $I = \int_0^{1/3} g(v) \, dv$,
we get the local expansion
\begin{equation*}
Q(1,1;t) = \frac{(4-I)}{3} \frac{\sqrt3}{2} \frac1{\sqrt{1-3t}} + O(1) ,
\end{equation*}
from which follows, by a transfer theorem, that
the number of walks of length~$n$ in Conjecture~\ref{conj:asympt-case-5}
is asymptotically equivalent to
$\frac{4 - I}{3} \frac{\sqrt3}{2 \sqrt\pi} \frac{3^n}{\sqrt n}$
provided~$I \neq 4$,
and the announced asymptotic formula of the conjecture
holds if and only if the constant identity~$I = 1$ holds.
\end{proof}

Again, the constant identity~$I = 1$ can be checked numerically,
and proving~$I \neq 4$ would be sufficient to prove
that the wanted asymptotic expansion
is of the form~$\kappa \frac{3^n}{\sqrt n}$.

\subsubsection{Other cases}

In all $19 \times 4$ cases,
our hypergeometric expressions for the counting functions produce
similar equalities of integrals.
While most of these constant equalities remain conjectural,
proving their correctness would imply the asymptotics in
Tables \ref{tab:kappa-in-F00t}, \ref{tab:kappa-in-F01t},
\ref{tab:kappa-in-F10t}, and~\ref{tab:kappa-in-F11t} in the Appendix.
In some cases, proving equality with limited numerical error
is sufficient to prove the constants $\rho$ and~$\gamma$
in the asymptotic behavior of the form~$\kappa \rho^n / n^\gamma$,
but not the constant~$\kappa$, which is just proved to be nonzero.

\subsection{Additional tables} \label{sec:computation}

The results of our computations are described in this section, where the
most characteristic features of the computed data are provided as tables.
For the data in full, we refer our readers to the web page of this article
\url{http://specfun.inria.fr/chyzak/ssw/}.

The hypergeometric series occurring in explicit expressions for~$Q(1,1;t)$
and for~$Q(x,y;t)$ are respectively given in Table~\ref{tab:2F1-in-Fxyt},
together with the rational-function substitution in those series.
The complete closed forms we obtained for $Q(0,0;t)$,
$Q(0,1;t)$, $Q(1,0;t)$, $Q(1,1;t)$, $Q(x,0;t)$, $Q(0,y;t)$, and~$Q(x,y;t)$
are given on the web site.

Various degrees, the differential order with respect to~$t$, and the
maximal size of integers in the annihilating operator and certificates that
we computed for each step set can be obtained from the web site as well.
Thus, the annihilators---the same as in~\cite{BK}---are finally proved.
Annihilators and certificates can also be downloaded in full form from the
web site, both as human-readable Maple sources and as a pre-digested Maple
library.

For each of the step set in Table~\ref{tab:steps}, the algebraic or
transcendental nature of the series $Q(0,0;t)$, $Q(0,1;t)$, $Q(1,0;t)$,
and~$Q(1,1;t)$, as well as asymptotic equivalent for the coefficients
of~$t^n$ in these series, are respectively given in Tables
\ref{tab:kappa-in-F00t}, \ref{tab:kappa-in-F01t}, \ref{tab:kappa-in-F10t},
and~\ref{tab:kappa-in-F11t} in the Appendix. There, `N'~denotes a transcendental series
and `Y'~an algebraic one, and the result has been proved by the method of
Section~\ref{sec:alg-or-trans}.  Regarding the status of asymptotic
equivalents, the constants in the tables (named $\kappa$ in
Section~\ref{sec:asymptotics}) refine (and correct) those provided
in~\cite{BoKa}.

The possibility of several interlaced regimes with same parameters $\rho$
and~$\gamma$ but non-zero different~$\kappa$ was first observed
and suggested to us by Melczer.  By redoing the calculations as
in~\cite{BoKa}, using convergence acceleration of (subsequences of) the
sequences $(q_{x,y;n})$ and~$(q_n)$ and using the PSLQ algorithm, we
numerically guessed those constants~$\kappa$, confirming the known ones and
obtaining the new ones.  That \cite{BoKa}~could overlook those constants is
explained by the nature of the method of convergence acceleration, which in
its native form considers values at powers of~2 only, and is thus not able
to distinguish regimes for odd vs even~$n$, let alone for various residues
modulo 3 or~4.
In their recent manuscript~\cite{MelczerWilson-2015-ALW},
Melczer and Wilson proved our guessed constants in all $19 \times 4$~cases
via analytic combinatorics in several variables.

\bibliographystyle{plain}

\section*{Appendix}

The tables on the next pages gather information
described in Section~\ref{sec:computation}.

\vfill

% Generated data.  Please, update generator rather than this file.

\begin{table}
\centerline{%
\begin{tabular}{|c|cccc|}
\hline
     & OEIS    & $\cS$            & alg & equiv
\\
\hline
 1 & A005568 & \stepset10101010 & N   & $\begin{cases} \frac{32 }{\pi } 
 \frac{4 ^n}{n^{3}} & (n=2p) \\ 0 & (n=2p+1) \end{cases}$ \\
 2 & A001246 & \stepset01010101 & N   & $\begin{cases} \frac{8 }{\pi } 
 \frac{4 ^n}{n^{3}} & (n=2p) \\ 0 & (n=2p+1) \end{cases}$ \\
 3 & A151362 & \stepset11011101 & N   & $\begin{cases} \frac{3 \sqrt{6 }}{\pi } 
 \frac{6 ^n}{n^{3}} & (n=2p) \\ 0 & (n=2p+1) \end{cases}$ \\
 4 & A172361 & \stepset11111111 & N   & $\frac{128 }{27 \pi } 
 \frac{8 ^n}{n^{3}}$ \\
 5 & A151332 & \stepset01001001 & N   & $\begin{cases} \frac{16 \sqrt{2 }}{\pi } 
 \frac{(2 \sqrt{2 }) ^n}{n^{3}} & (n=4p) \\ 0 & (n=4p+1) \\ 0 & (n=4p+2) \\ 0 & (n=4p+3) \end{cases}$ \\
 6 & A151357 & \stepset01101011 & N   & $\frac{2 A ^{3/2} }{\pi } 
 \frac{(2 A ) ^n}{n^{3}}$ \\
 7 & A151341 & \stepset11001001 & N   & $\begin{cases} \frac{12 \sqrt{3 }}{\pi } 
 \frac{(2 \sqrt{3 }) ^n}{n^{3}} & (n=2p) \\ 0 & (n=2p+1) \end{cases}$ \\
 8 & A151368 & \stepset11101011 & N   & $\frac{2 B ^{3/2} }{\pi } 
 \frac{(2 B ) ^n}{n^{3}}$ \\
 9 & A151345 & \stepset11010101 & N   & $\begin{cases} \frac{24 \sqrt{30 }}{25 \pi } 
 \frac{(2 \sqrt{6 }) ^n}{n^{3}} & (n=2p) \\ 0 & (n=2p+1) \end{cases}$ \\
10 & A151370 & \stepset11110111 & N   & $\frac{2 \mu ^{3} C ^{3/2} }{\pi } 
 \frac{(2 C ) ^n}{n^{3}}$ \\
11 & A151341 & \stepset10011100 & N   & $\begin{cases} \frac{12 \sqrt{3 }}{\pi } 
 \frac{(2 \sqrt{3 }) ^n}{n^{3}} & (n=2p) \\ 0 & (n=2p+1) \end{cases}$ \\
12 & A151368 & \stepset10111110 & N   & $\frac{2 B ^{3/2} }{\pi } 
 \frac{(2 B ) ^n}{n^{3}}$ \\
13 & A151345 & \stepset01011101 & N   & $\begin{cases} \frac{24 \sqrt{30 }}{25 \pi } 
 \frac{(2 \sqrt{6 }) ^n}{n^{3}} & (n=2p) \\ 0 & (n=2p+1) \end{cases}$ \\
14 & A151370 & \stepset01111111 & N   & $\frac{2 \mu ^{3} C ^{3/2} }{\pi } 
 \frac{(2 C ) ^n}{n^{3}}$ \\
15 & A151332 & \stepset10010100 & N   & $\begin{cases} \frac{16 \sqrt{2 }}{\pi } 
 \frac{(2 \sqrt{2 }) ^n}{n^{3}} & (n=4p) \\ 0 & (n=4p+1) \\ 0 & (n=4p+2) \\ 0 & (n=4p+3) \end{cases}$ \\
16 & A151357 & \stepset10110110 & N   & $\frac{2 A ^{3/2} }{\pi } 
 \frac{(2 A ) ^n}{n^{3}}$ \\
17 & A151334 & \stepset10010010 & N   & $\begin{cases} \frac{81 \sqrt{3 }}{\pi } 
 \frac{3 ^n}{n^{4}} & (n=3p) \\ 0 & (n=3p+1) \\ 0 & (n=3p+2) \end{cases}$ \\
18 & A151366 & \stepset10111011 & N   & $\frac{27 \sqrt{3 }}{\pi } 
 \frac{6 ^n}{n^{4}}$ \\
19 & A138349 & \stepset00110011 & N   & $\begin{cases} \frac{768 }{\pi } 
 \frac{4 ^n}{n^{5}} & (n=2p) \\ 0 & (n=2p+1) \end{cases}$ \\
\hline
\end{tabular}
}
\caption{Nature of the generating series for $(x,y)=(0,0)$ and coefficient
  asymptotics}
\centering\small $A=1+\sqrt2$, \ $B=1+\sqrt3$, \ $C=1+\sqrt6$, \ $\lambda=7+3\sqrt6$, \ $\mu=\sqrt{\frac{4\sqrt6-1}{19}}$
\label{tab:kappa-in-F00t}
\end{table}

\begin{table}
\centerline{%
\begin{tabular}{|c|cccc|}
\hline
     & OEIS    & $\cS$            & alg & equiv
\\
\hline
 1 & A005558 & \stepset10101010 & N   & $\frac{8 }{\pi } 
 \frac{4 ^n}{n^{2}}$ \\
 2 & A151392 & \stepset01010101 & N   & $\begin{cases} \frac{4 }{\pi } 
 \frac{4 ^n}{n^{2}} & (n=2p) \\ 0 & (n=2p+1) \end{cases}$ \\
 3 & A151478 & \stepset11011101 & N   & $\frac{3 \sqrt{6 }}{2 \pi } 
 \frac{6 ^n}{n^{2}}$ \\
 4 & A151496 & \stepset11111111 & N   & $\frac{32 }{9 \pi } 
 \frac{8 ^n}{n^{2}}$ \\
 5 & A151380 & \stepset01001001 & N   & $\frac{3 }{4 } \sqrt{\frac{3 }{\pi }} 
 \frac{3 ^n}{n^{3/2}}$ \\
 6 & A151450 & \stepset01101011 & N   & $\frac{5 }{16 } \sqrt{\frac{10 }{\pi }} 
 \frac{5 ^n}{n^{3/2}}$ \\
 7 & A148790 & \stepset11001001 & N   & $\frac{8 }{3 \sqrt{\pi }} 
 \frac{4 ^n}{n^{3/2}}$ \\
 8 & A151485 & \stepset11101011 & N   & $\sqrt{\frac{3 }{\pi }} 
 \frac{6 ^n}{n^{3/2}}$ \\
 9 & A151440 & \stepset11010101 & N   & $\frac{5 }{24 } \sqrt{\frac{10 }{\pi }} 
 \frac{5 ^n}{n^{3/2}}$ \\
10 & A151493 & \stepset11110111 & N   & $\frac{7 }{54 } \sqrt{\frac{21 }{\pi }} 
 \frac{7 ^n}{n^{3/2}}$ \\
11 & A151394 & \stepset10011100 & N   & $\begin{cases} \frac{36 \sqrt{3 }}{\pi } 
 \frac{(2 \sqrt{3 }) ^n}{n^{3}} & (n=2p) \\ \frac{54 }{\pi } 
 \frac{(2 \sqrt{3 }) ^n}{n^{3}} & (n=2p+1) \end{cases}$ \\
12 & A151472 & \stepset10111110 & N   & $\frac{3 B ^{7/2} }{2 \pi } 
 \frac{(2 B ) ^n}{n^{3}}$ \\
13 & A151437 & \stepset01011101 & N   & $\begin{cases} \frac{72 \sqrt{30 }}{5 \pi } 
 \frac{(2 \sqrt{6 }) ^n}{n^{3}} & (n=2p) \\ \frac{864 \sqrt{5 }}{25 \pi } 
 \frac{(2 \sqrt{6 }) ^n}{n^{3}} & (n=2p+1) \end{cases}$ \\
14 & A151492 & \stepset01111111 & N   & $\frac{6 \lambda \mu ^{3} C ^{5/2} }{5 \pi } 
 \frac{(2 C ) ^n}{n^{3}}$ \\
15 & A151375 & \stepset10010100 & N   & $\begin{cases} \frac{448 \sqrt{2 }}{9 \pi } 
 \frac{(2 \sqrt{2 }) ^n}{n^{3}} & (n=4p) \\ \frac{640 }{9 \pi } 
 \frac{(2 \sqrt{2 }) ^n}{n^{3}} & (n=4p+1) \\ \frac{416 \sqrt{2 }}{9 \pi } 
 \frac{(2 \sqrt{2 }) ^n}{n^{3}} & (n=4p+2) \\ \frac{512 }{9 \pi } 
 \frac{(2 \sqrt{2 }) ^n}{n^{3}} & (n=4p+3) \end{cases}$ \\
16 & A151430 & \stepset10110110 & N   & $\frac{4 A ^{7/2} }{\pi } 
 \frac{(2 A ) ^n}{n^{3}}$ \\
17 & A151378 & \stepset10010010 & N   & $\frac{27 }{8 } \sqrt{\frac{3 }{\pi }} 
 \frac{3 ^n}{n^{5/2}}$ \\
18 & A151483 & \stepset10111011 & Y   & $\frac{27 }{8 } \sqrt{\frac{3 }{\pi }} 
 \frac{6 ^n}{n^{5/2}}$ \\
19 & A005568 & \stepset00110011 & N   & $\begin{cases} \frac{32 }{\pi } 
 \frac{4 ^n}{n^{3}} & (n=2p) \\ 0 & (n=2p+1) \end{cases}$ \\
\hline
\end{tabular}
}
\caption{Nature of the generating series for $(x,y)=(0,1)$ and coefficient
  asymptotics}
\centering\small $A=1+\sqrt2$, \ $B=1+\sqrt3$, \ $C=1+\sqrt6$, \ $\lambda=7+3\sqrt6$, \ $\mu=\sqrt{\frac{4\sqrt6-1}{19}}$
\label{tab:kappa-in-F01t}
\end{table}

\begin{table}
\centerline{%
\begin{tabular}{|c|cccc|}
\hline
     & OEIS    & $\cS$            & alg & equiv
\\
\hline
 1 & A005558 & \stepset10101010 & N   & $\frac{8 }{\pi } 
 \frac{4 ^n}{n^{2}}$ \\
 2 & A151392 & \stepset01010101 & N   & $\begin{cases} \frac{4 }{\pi } 
 \frac{4 ^n}{n^{2}} & (n=2p) \\ 0 & (n=2p+1) \end{cases}$ \\
 3 & A151471 & \stepset11011101 & N   & $\begin{cases} \frac{2 \sqrt{6 }}{\pi } 
 \frac{6 ^n}{n^{2}} & (n=2p) \\ 0 & (n=2p+1) \end{cases}$ \\
 4 & A151496 & \stepset11111111 & N   & $\frac{32 }{9 \pi } 
 \frac{8 ^n}{n^{2}}$ \\
 5 & A151379 & \stepset01001001 & N   & $\begin{cases} \frac{4 \sqrt{2 }}{\pi } 
 \frac{(2 \sqrt{2 }) ^n}{n^{2}} & (n=2p) \\ 0 & (n=2p+1) \end{cases}$ \\
 6 & A148934 & \stepset01101011 & N   & $\frac{\sqrt{2 }A ^{3/2} }{\pi } 
 \frac{(2 A ) ^n}{n^{2}}$ \\
 7 & A151410 & \stepset11001001 & N   & $\begin{cases} \frac{4 \sqrt{3 }}{\pi } 
 \frac{(2 \sqrt{3 }) ^n}{n^{2}} & (n=2p) \\ 0 & (n=2p+1) \end{cases}$ \\
 8 & A151464 & \stepset11101011 & N   & $\frac{2 B ^{3/2} \sqrt{3 }}{3 \pi } 
 \frac{(2 B ) ^n}{n^{2}}$ \\
 9 & A151423 & \stepset11010101 & N   & $\begin{cases} \frac{4 \sqrt{30 }}{5 \pi } 
 \frac{(2 \sqrt{6 }) ^n}{n^{2}} & (n=2p) \\ 0 & (n=2p+1) \end{cases}$ \\
10 & A151490 & \stepset11110111 & N   & $\frac{\sqrt{6 }\mu C ^{3/2} }{3 \pi } 
 \frac{(2 C ) ^n}{n^{2}}$ \\
11 & A151410 & \stepset10011100 & N   & $\begin{cases} \frac{4 \sqrt{3 }}{\pi } 
 \frac{(2 \sqrt{3 }) ^n}{n^{2}} & (n=2p) \\ 0 & (n=2p+1) \end{cases}$ \\
12 & A151464 & \stepset10111110 & N   & $\frac{2 B ^{3/2} \sqrt{3 }}{3 \pi } 
 \frac{(2 B ) ^n}{n^{2}}$ \\
13 & A151423 & \stepset01011101 & N   & $\begin{cases} \frac{4 \sqrt{30 }}{5 \pi } 
 \frac{(2 \sqrt{6 }) ^n}{n^{2}} & (n=2p) \\ 0 & (n=2p+1) \end{cases}$ \\
14 & A151490 & \stepset01111111 & N   & $\frac{\sqrt{6 }\mu C ^{3/2} }{3 \pi } 
 \frac{(2 C ) ^n}{n^{2}}$ \\
15 & A151379 & \stepset10010100 & N   & $\begin{cases} \frac{4 \sqrt{2 }}{\pi } 
 \frac{(2 \sqrt{2 }) ^n}{n^{2}} & (n=2p) \\ 0 & (n=2p+1) \end{cases}$ \\
16 & A148934 & \stepset10110110 & N   & $\frac{\sqrt{2 }A ^{3/2} }{\pi } 
 \frac{(2 A ) ^n}{n^{2}}$ \\
17 & A151497 & \stepset10010010 & N   & $\frac{27 }{8 } \sqrt{\frac{3 }{\pi }} 
 \frac{3 ^n}{n^{5/2}}$ \\
18 & A151483 & \stepset10111011 & Y   & $\frac{27 }{8 } \sqrt{\frac{3 }{\pi }} 
 \frac{6 ^n}{n^{5/2}}$ \\
19 & A005817 & \stepset00110011 & N   & $\frac{32 }{\pi } 
 \frac{4 ^n}{n^{3}}$ \\
\hline
\end{tabular}
}
\caption{Nature of the generating series for $(x,y)=(1,0)$ and coefficient
  asymptotics}
\centering\small $A=1+\sqrt2$, \ $B=1+\sqrt3$, \ $C=1+\sqrt6$, \ $\lambda=7+3\sqrt6$, \ $\mu=\sqrt{\frac{4\sqrt6-1}{19}}$
\label{tab:kappa-in-F10t}
\end{table}

\begin{table}
\centerline{%
\begin{tabular}{|c|cccc|}
\hline
     & OEIS    & $\cS$            & alg & equiv
\\
\hline
 1 & A005566 & \stepset10101010 & N   & $\frac{4 }{\pi } 
 \frac{4 ^n}{n}$ \\
 2 & A018224 & \stepset01010101 & N   & $\frac{2 }{\pi } 
 \frac{4 ^n}{n}$ \\
 3 & A151312 & \stepset11011101 & N   & $\frac{\sqrt{6 }}{\pi } 
 \frac{6 ^n}{n}$ \\
 4 & A151331 & \stepset11111111 & N   & $\frac{8 }{3 \pi } 
 \frac{8 ^n}{n}$ \\
 5 & A151266 & \stepset01001001 & N   & $\frac1{2 } \sqrt{\frac{3 }{\pi }} 
 \frac{3 ^n}{n^{1/2}}$ \\
 6 & A151307 & \stepset01101011 & N   & $\frac1{4 } \sqrt{\frac{10 }{\pi }} 
 \frac{5 ^n}{n^{1/2}}$ \\
 7 & A151291 & \stepset11001001 & N   & $\frac{4 }{3 \sqrt{\pi }} 
 \frac{4 ^n}{n^{1/2}}$ \\
 8 & A151326 & \stepset11101011 & N   & $\frac{2 }{3 } \sqrt{\frac{3 }{\pi }} 
 \frac{6 ^n}{n^{1/2}}$ \\
 9 & A151302 & \stepset11010101 & N   & $\frac1{6 } \sqrt{\frac{10 }{\pi }} 
 \frac{5 ^n}{n^{1/2}}$ \\
10 & A151329 & \stepset11110111 & N   & $\frac1{9 } \sqrt{\frac{21 }{\pi }} 
 \frac{7 ^n}{n^{1/2}}$ \\
11 & A151261 & \stepset10011100 & N   & $\begin{cases} \frac{12 \sqrt{3 }}{\pi } 
 \frac{(2 \sqrt{3 }) ^n}{n^{2}} & (n=2p) \\ \frac{18 }{\pi } 
 \frac{(2 \sqrt{3 }) ^n}{n^{2}} & (n=2p+1) \end{cases}$ \\
12 & A151297 & \stepset10111110 & N   & $\frac{\sqrt{3 }B ^{7/2} }{2 \pi } 
 \frac{(2 B ) ^n}{n^{2}}$ \\
13 & A151275 & \stepset01011101 & N   & $\begin{cases} \frac{12 \sqrt{30 }}{\pi } 
 \frac{(2 \sqrt{6 }) ^n}{n^{2}} & (n=2p) \\ \frac{144 \sqrt{5 }}{5 \pi } 
 \frac{(2 \sqrt{6 }) ^n}{n^{2}} & (n=2p+1) \end{cases}$ \\
14 & A151314 & \stepset01111111 & N   & $\frac{\sqrt{6 }\lambda \mu C ^{5/2} }{5 \pi } 
 \frac{(2 C ) ^n}{n^{2}}$ \\
15 & A151255 & \stepset10010100 & N   & $\begin{cases} \frac{24 \sqrt{2 }}{\pi } 
 \frac{(2 \sqrt{2 }) ^n}{n^{2}} & (n=4p) \\ \frac{32 }{\pi } 
 \frac{(2 \sqrt{2 }) ^n}{n^{2}} & (n=4p+1) \\ \frac{24 \sqrt{2 }}{\pi } 
 \frac{(2 \sqrt{2 }) ^n}{n^{2}} & (n=4p+2) \\ \frac{32 }{\pi } 
 \frac{(2 \sqrt{2 }) ^n}{n^{2}} & (n=4p+3) \end{cases}$ \\
16 & A151287 & \stepset10110110 & N   & $\frac{2 \sqrt{2 }A ^{7/2} }{\pi } 
 \frac{(2 A ) ^n}{n^{2}}$ \\
17 & A001006 & \stepset10010010 & Y   & $\frac{3 }{2 } \sqrt{\frac{3 }{\pi }} 
 \frac{3 ^n}{n^{3/2}}$ \\
18 & A129400 & \stepset10111011 & Y   & $\frac{3 }{2 } \sqrt{\frac{3 }{\pi }} 
 \frac{6 ^n}{n^{3/2}}$ \\
19 & A005558 & \stepset00110011 & N   & $\frac{8 }{\pi } 
 \frac{4 ^n}{n^{2}}$ \\
\hline
\end{tabular}
}
\caption{Nature of the generating series for $(x,y)=(1,1)$ and coefficient
  asymptotics}
\centering\small $A=1+\sqrt2$, \ $B=1+\sqrt3$, \ $C=1+\sqrt6$, \ $\lambda=7+3\sqrt6$, \ $\mu=\sqrt{\frac{4\sqrt6-1}{19}}$
\label{tab:kappa-in-F11t}
\end{table}

\end{document}